\documentclass{amsart}
\pdfoutput=1
\usepackage{color}
\usepackage{latexsym}
\usepackage[nointegrals]{wasysym}
\usepackage{amssymb,amsrefs}
\usepackage{amsmath,amsthm,fancyhdr,pifont}
\usepackage{mathrsfs}
\usepackage{graphicx}
\usepackage{verbatim}
\usepackage{pb-diagram,amscd,amsfonts,amsmath}
\usepackage{hyperref}

\usepackage[all]{xy}
\usepackage{setspace} 

\newtheorem{theorem}{Theorem}[section]

\newtheorem{lemma}[theorem]{Lemma}

\newtheorem{corollary}[theorem]{Corollary}
\newtheorem{proposition}[theorem]{Proposition}

\newtheorem{thm}[theorem]{Theorem}

\newtheorem{prop}[theorem]{Proposition}

\newenvironment{pf*}[1]{\proof[#1]}{\endproof}
\usepackage{euscript}

\theoremstyle{definition}
\newtheorem{definition}{Definition}[section]
\newtheorem{defn}{Definition}[section]

\newcommand{\Forget}{\operatorname{Forget}}
\newcommand{\Push}{\operatorname{Push}}
\newcommand{\Lift}{\operatorname{Lift}}
\newcommand{\PMCG}{{\text{PMCG}}}
\newcommand{\RMCG}{{\text{RMCG}}}
\newcommand{\MCG}{{\text{MCG}}}

\theoremstyle{remark}
\newtheorem{remark}[theorem]{Remark}

\theoremstyle{remark}
\newtheorem{rem}{Remark}[section]

\renewcommand{\deg}{\operatorname{deg}}
\newcommand{\riem}{\hat{\CC}}

\renewcommand{\mod}{\operatorname{mod}}
\newcommand{\tl}{\tilde}
\newcommand{\wtl}{\widetilde}

\renewcommand\O{\mathcal{O}}

\newcommand\OO{\wasylozenge}

\newcommand\Q{\mathbb{Q}}

\newcommand\N{\mathbb{N}}

\newcommand\Z{\mathbb{Z}}
\newcommand\R{\mathbb{R}}

\numberwithin{equation}{section}

\newcommand{\thmref}[1]{Theorem~\ref{#1}}
\newcommand{\propref}[1]{Proposition~\ref{#1}}

\newcommand{\corref}[1]{Corollary~\ref{#1}}

\newcommand{\cA}{{\mathcal A}}

\newcommand{\cM}{{\mathcal M}}

\newcommand{\cF}{{\mathcal F}}
\newcommand{\cG}{{\mathcal G}}

\newcommand{\cT}{{\mathcal T}}

\newcommand{\cP}{{\mathcal P}}

\newcommand{\cH}{{\mathcal H}}

\newcommand{\CC}{{\mathbb C}}

\newcommand{\RR}{{\mathbb R}}

\newcommand{\TT}{{\mathbb T}}
\newcommand{\ZZ}{{\mathbb Z}}
\newcommand{\NN}{{\mathbb N}}
\newcommand{\DD}{{\mathbb D}}
\newcommand{\HH}{{\mathbb H}}

{\bf}{\it}
{\bf}{\it}
{\bf}{\it}
{\bf}{\it}
\newtheorem*{thurstonthm}{Thurston's Theorem}{\bf}{\it}

\def\Int{\text{int}}

\def\Pt{{\tilde{P}}}

\renewcommand\S{\mathcal{S}}

\newcommand\Sphere{{\mathbb{S}^2}}
\newcommand\PP{Q_f}
\newcommand{\id}{\mbox{\rm id}}
\renewcommand{\mod}{\mbox{\rm mod }}

\newcommand{\fh}{\hat{f}}

\newcommand{\sm}{\setminus}
\newcommand{\eps}{\varepsilon}

{\par\noindent\textit{Proof of (#1)}%
} 

\begin{document}

\title[]{Constructive geometrization of Thurston maps and decidability of Thurston equivalence}
\date{October 5, 2013}

\author{Nikita Selinger and  Michael Yampolsky}
\begin{abstract}
The key result in the present paper is a direct analogue of the celebrated Thurston's Theorem \cite{DH}
for marked Thurston maps with parabolic orbifolds. Combining this result with previously developed techniques,
we prove that every Thurston map can be constructively geometrized in a canonical fashion. As a consequence,
we give a partial resolution of the general problem of decidability of Thurston equivalence of two postcritically
finite branched covers of $S^2$ (cf. \cite{BBY}).
\end{abstract}

\maketitle

\section{Introduction}
\label{section-intro}

A Thurston map is a basic object of study in one-dimensional dynamics: a branched covering $f$ of the $2$-sphere with finite
critical orbits. Such a map can be described in a purely combinatorial language by introducing a suitable triangulation of $S^2$ whose
set of vertices includes the critical orbits of $f$. Different combinatorial descriptions of the map lead to a natural {\it combinatorial}
or {\it Thurston equivalence} relation. A natural question arises whether given two such combinatorial objects, it can be decided if they are equivalent or not in some systematic, i.e. algorithmic, fashion. 

We briefly outline the history of the problem.
 A central theorem in the subject is the result of Thurston \cite{DH} that describes, in a topological
language, which Thurston maps are combinatorially equivalent to rational mappings of $\riem$. 
 In the case when an equivalent rational mapping
exists, it is essentially unique, and the proof of the theorem \cite{DH} supplies an iterative algorithm for approximating its coefficients.
The only obstacle for the existence of a Thurston equivalent rational map is the presence of a {\it Thurston obstruction} which is a 
finite collection of curves in $S^2$ that satisfies a certain combinatorial inequality. 
Equivalence to a rational mapping can thus  be seen as a {\it geometrization} of the branched covering: equipping the topological object with
a canonical geometric description.

In \cite{BBY} it was shown that,  outside of some exceptional cases, the question of Thurston equivalence to a rational mapping is 
{\it algorithmically decidable}. Namely, there exists an algorithm $\cA_1$ which, given a combinatorial description of $f$, outputs 
$1$ if $f$ is equivalent to a rational mapping and $0$ otherwise. Moreover, in the former case, $\cA_1$ identifies the rational mapping.
Since two different rational mappings are easy to distinguish -- for instance, by comparing their coefficients after some normalization -- this implies
that in the case when either $f$ or $g$ has no Thurston obstruction, the statement of the Main Theorem I can be deduced from the existence of $\cA_1$ \cite{BBY}.

Our work  concentrates on the situation when Thurston maps are obstructed. In this case, geometrization may be achieved by {\it decomposition}
into geometrizable components \cite{Pil2}. 
We show:

\medskip
\noindent
{\bf Main Theorem I.} 
{\it Every Thruston map admits a constructive canonical geometrization.}

\medskip
\noindent
The main step in the proof is a direct analogue of Thurston's Theorem for the exceptional cases, Thurston maps with  parabolic orbifolds:

\medskip
\noindent
{\bf Main Theorem II.} {\it A marked Thurston map with parabolic orbifold is geometrizable if and only if it has no degenerate Levy cycles.}

\medskip
\noindent
Detailed versions of both statements will be given below, after some preliminaries. As a consequence we obtain a partial resolution of the 
general question of decidability of Thurston equivalence:

\medskip
\noindent
{\bf Main Theorem III.} {\it 
There exists an algorithm $\cA$ which does the following. Let $f$ and $g$ be marked Thurston maps and assume that every element of the
 canonical geometrization of $f$ has hyperbolic orbifold. The algorithm $\cA$, given the combinatorial descriptions of $f$ and $g$,
outputs $1$ if $f$ and $g$ are Thurston equivalent and $0$ otherwise.
}

\medskip
\noindent

\section{Geometric preliminaries}
\subsection*{Mapping Class Groups}
When we talk about  a 
surface with holes, we will always mean a surface $S$ with boundary,
 which is obtained from a surface without holes by removing a collection of disjoint Jordan disks.
A 
surface $S$  is of  {\it finite topological type} if it is a genus $g$ surface with $m$ holes and $n$ punctures,
where $g,m,n<\infty$. The {\it Mapping Class Group} $\MCG(S)$ is defined as the group of homeomorphisms $S\to S$
which restrict to the identity on $\partial S$, up to isotopy relative $\partial S$.

The elements of $\MCG(S)$ are allowed to interchange the punctures of $S$; if we further restrict to homeomorphisms
which fix each puncture individually, we obtain the {\it pure Mapping Class Group} $\PMCG(S)$. If we denote by $\Sigma_n$ the group of 
permutations of $n$ elements (punctures, in our case), then we have a short exact sequence
$$1\longrightarrow \PMCG(S_{g,r}^n)\longrightarrow \MCG(S_{g,r}^n)\longrightarrow \Sigma_n\longrightarrow 1.$$

We refer the reader to \cite{primer} for a detailed discussion of Mapping Class Groups. 
Replacing homeomorphisms with diffeomorphisms, and/or isotopy with homotopy leads to an equivalent definition of $\MCG(S)$. 

Throughout the article, we denote by $T_\gamma$ the Dehn twist around a curve $\gamma$. We use the following fact:
\begin{prop}
\label{pmod}
The group $\PMCG(S^{2},P)$ is generated by a finite number of explicit Dehn twists.

\end{prop}
The finiteness of the number of generating twists is a classical result of Dehn; Lickorish \cite{Lic} has made the construction explicit. 
See, for example, \cite{primer} for an exposition.

\section{Thurston maps}
\label{section1}
In this section we recall the basic setting of Thurston's characterization of rational functions.
\subsection{Branched covering maps}
\label{section1.1}
Let $f:S^{2}\rightarrow S^{2}$ be an orientation-preserving branched covering self-map of the two-sphere. We define the \textit{postcritical set $P_{f}$} by
\[
 P_{f}:=\bigcup_{n>0}f^{\circ n}(\Omega_{f}),
\]
where $\Omega_{f}$ is the set of critical points of $f$. When the postcritical set $P_{f}$ is finite, we say that $f$ is \textit{postcritically finite}.

A {\it (marked) Thurston map} is a pair $(f,Q_f)$ where $f:S^2\to S^2$ is a postcritically finite ramified covering of degree at least 2 and $Q_f$ is a finite collection of marked points
$Q_f\subset S^2$ which contains $P_f$ and is $f$-invariant: $f(Q_f)\subset Q_f$. Thus, all elements of $Q_f$ are pre-periodic for $f$. 
\paragraph*{\bf Thurston equivalence.} 
Two marked Thurston maps $(f,Q_f)$ and $(g,Q_g)$ are \textit{Thurston equivalent} if there are homeomorphisms $\phi_{0}, \phi_{1}:S^{2}\rightarrow S^{2}$ such that
\begin{enumerate}
 \item the maps $\phi_{0}, \phi_{1}$ coincide on $Q_f$, send $Q_{f}$ to $Q_{g}$ and   and are isotopic \text{rel } $Q_f$; 
\item the diagram
\[
\begin{CD}
S^{2} @>\phi_{1}>> S^{2}\\
@VVfV @VVgV\\
S^{2} @>\phi_{0}>> S^{2}
\end{CD}
\]
commutes.
\end{enumerate}
 
\subsection*{Orbifold of a Thurston map} Given a Thurston map $f:S^{2} \rightarrow S^{2}$, we define a function $N_{f}:S^{2} \rightarrow \mathbb{N}\cup {\infty}$ as follows:
\[
N_{f}(x)=\begin{cases}
1& \text{if $x \notin P_{f}$},\\
\infty & \text{if $x$ is in a cycle containing a critical point},\\
\underset{f^{k}(y)=x}{\text{lcm}} \text{deg}_{y}(f^{\circ k}) & \text{ otherwise}.
\end{cases}
\]

The pair $(S^{2},N_{f})$ is called the \textit{orbifold of $f$}. 
The {\it signature} of the orbifold $(S^2,N_f)$ is the set $\{N_f(x)\text{ for }x\text{ such that }1<N_f(x)<\infty\}$. 
The \textit{Euler characteristic} of the orbifold is given by 
\begin{equation}
\label{eq:euler}
 \chi(S^{2},N_{f}):= 2-\sum_{x \in P_{f}}\left(1-\frac{1}{N_{f}(x)}\right).
\end{equation}
One can prove that $\chi(S^{2},N_{f})\leq 0$. In the case where $\chi(S^{2},N_{f})< 0$, we say that the orbifold is \textit{hyperbolic}. Observe that most orbifolds are hyperbolic: indeed, as soon as the cardinality $\vert P_{f} \vert >4$, the orbifold is hyperbolic.

\subsection*{Thurston maps with parabolic orbifolds}

A complete classification of postcritically finite branched covers  with parabolic orbifolds has been given in \cite{DH}. All postcritically finite rational functions with parabolic orbifolds have been extensively described in \cite{milnorlattes}.  In this section, we remind the reader of basic results on Thurston maps with parabolic orbifolds.

Recall that a map $f \colon (S_1,v_1) \to (S_2,v_2)$ is a covering map of orbifolds if $$v_1(x)\deg_x f = v_2(f(x))$$ for any $x \in S_1$. The following proposition is found in \cite{DH}:

\begin{proposition} 
\begin{itemize}
	\item[i.] 
If $f \colon \Sphere \to \Sphere$ is a postcritically finite branched cover, then $\chi(O_f) \le 0$.
\item[ii.] If $\chi(O_f) = 0$, then $f \colon O_f \to O_f$ is a covering map of orbifolds.
\end{itemize}
\end{proposition}

Equation~(\ref{eq:euler}) gives six possibilities for $\chi(O_f)=0$. If we record all the values of $v_f$ that are bigger than 1, we get one of the following orbifold signatures.

\begin{enumerate}
	\item $(\infty,\infty),$
	\item $(2,2,\infty),$
	\item $(2,4,4),$
	\item $(2,3,6),$
	\item $(3,3,3),$
	\item $(2,2,2,2).$
\end{enumerate}

\noindent
In cases (1)-(5) the orbifolds have a unique complex structure, and can be realized as a quotient $\CC/ G$ of the complex plane by a discrete
group of automorphisms $ G$ as follows (cf. \cite{DH}):
\begin{enumerate}
\item $ G=<z\mapsto z+1>$,
\item $ G=<z\mapsto z+1,z\mapsto -z>$,
\item $ G=<z\mapsto z+a, z\mapsto iz>$, where $a\in\ZZ[i]$,
\item $ G=<z\mapsto z+a, z\mapsto wz>$, where $w=e^{i\pi/3},\; a\in\ZZ[w]$,
\item $ G=<z\mapsto z+a, z\mapsto w^2z>$, where $w=e^{i\pi/3},\; a\in\ZZ[w]$.
\end{enumerate}

We are mostly interested in the last case. 
We will refer to a Thurston map that has orbifold with signature $(2,2,2,2)$ simply as a \emph{$(2,2,2,2)$-map}. 
An orbifold with signature $(2,2,2,2)$ is a quotient of a torus $T$ by an involution $i$; 
the four fixed points of the involution $i$ correspond to the points with ramification weight 2 on the orbifold. 
The corresponding branched cover $P:T\to\Sphere$ has exactly 4 simple critical points which are the fixed points of $i$. 
It follows that a $(2,2,2,2)$-map $f$ can be lifted to a covering self-map $\fh$ of $T$.

An orbifold with signature $(2,2,2,2)$  has a unique affine structure of the quotient
$\RR^2/ G$ where 
$$ G=<z\mapsto z+1,z\mapsto z+i, z\mapsto -z>.$$
We will denote this quotient by the symbol $\OO$, which graphically represents a ``pillowcase'' -- a sphere
with four corner points.

An important example of a $(2,2,2,2)$-map is a {\it flexible Latt{\'e}s rational map} constructed as follows.
Let
$$\TT\simeq \TT_\Lambda=\CC/\Lambda$$ where the lattice $\Lambda=<1,\tau>$, with $\tau\in\HH$. Set $i(z)=-z$. Then $\TT_\lambda/i\simeq \hat\CC$ and
the branched cover 
$$\wp:\CC\to\CC/\Lambda\to \hat\CC$$
is the Weierstrass elliptic function $\wp$ with periods $1$, $\tau$. 
Consider the parallelogram $P$ with vertices $0$, $1$, $\tau$, and $1+\tau$ which is the fundamental domain of $\Lambda$. The four simple fixed points of
the involution $i$ are the $\wp$-images of $0$, $1/2$, $\tau/2$ and $(1+\tau)/2$. They are the critical points of the brached cover
$\TT_\Lambda\to \hat\CC$.

\begin{figure}[ht]
\includegraphics[width=\textwidth]{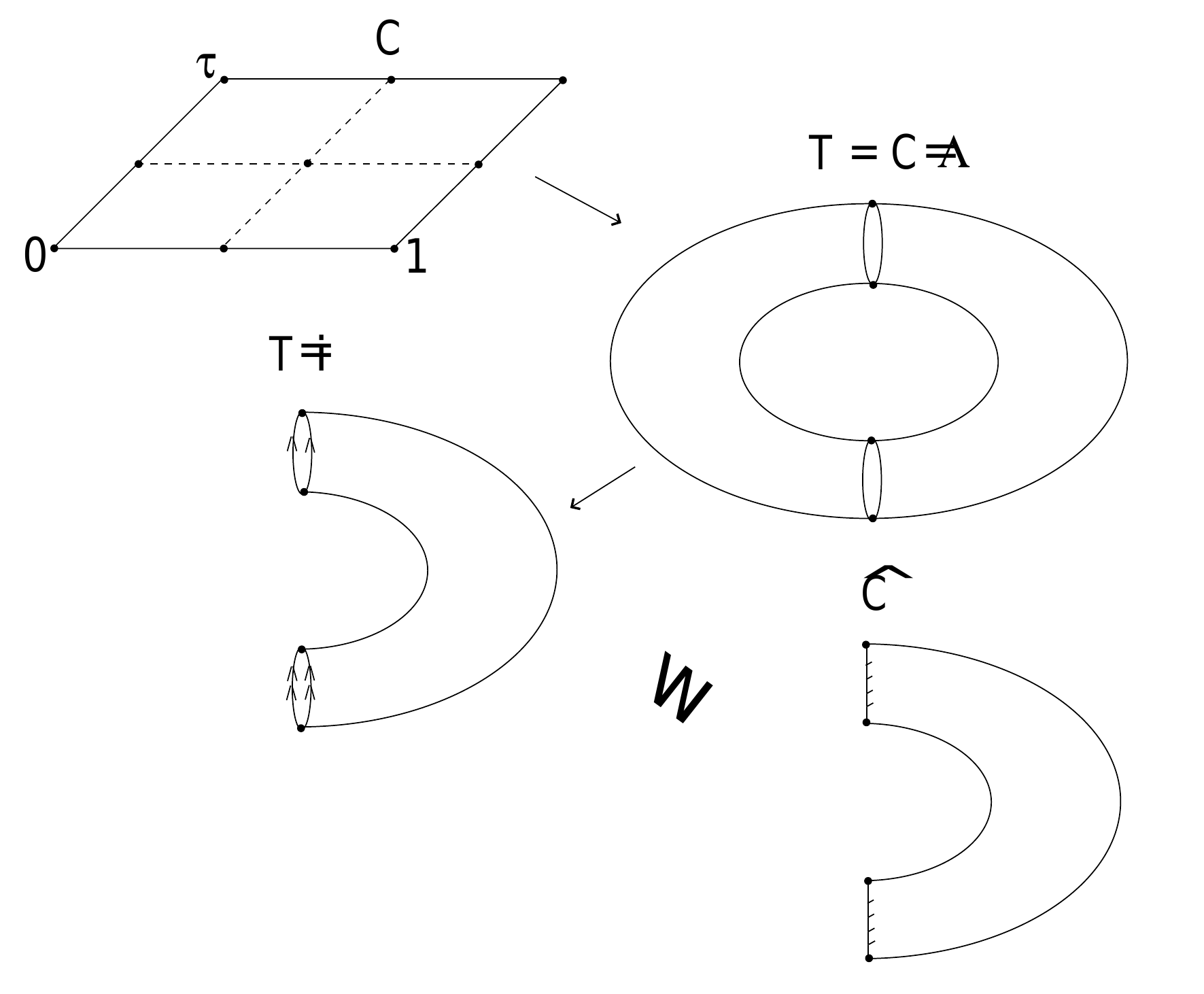}
\caption{\label{fig-weierstrass}Illustration of the branched cover $\wp$. The critical points of $\wp$
are marked in a fundamental parallelogram of the lattice $\Lambda$, as well as their images.}
\end{figure}

Set $$A(z)\equiv az+b,\text{ where }a\in\ZZ\text{ with }|a|>1\text{, and }b=(m+n\tau)/2\in\Lambda/2.$$
The complex-affine map $A$ projects to a well-defined rational map 
$$L:\TT_\Lambda\to\TT_\Lambda$$ of degree $a^2$.
Trivially, all of the postcritical set of $L$ lies in the projection of $\Lambda/2$ in $\hat\CC$ and hence is finite.
Note that as long as the values of $a$, $m$, and $n$ are the same, two different maps $L$ are topologically conjugate 
for all values of $\tau$. In particular, they cannot be distinguished by Thurston equivalence, which shows that the uniqueness part of 
Thurston's Theorem does not generally
hold in the parabolic orbifold case.

As before, let $f$ be a $(2,2,2,2)$-map, and $p:T\to\S^2$.
Take any simple closed  curve $\gamma$ on $\Sphere \setminus \PP$. Then $p^{-1}(\gamma)$ has either one or two components that are simple closed curves. 

The following propositions are straightforward (see, for example, \cite{seljmd}):
\begin{proposition}
 If there are exactly two postcritical points of $f$ in each complementary component of $\gamma$, then the $p$-preimage of $\gamma$ consists of two components that are homotopic in $T$ and non-trivial in $H_1(T,\Z)$. Otherwise, all preimages of $\gamma$ are trivial.
 \label{prop:trivialhom}
\end{proposition}

Every homotopy class of simple closed curves $\gamma$ on $T$ defines, up to sign, an element $\langle \gamma \rangle$ of $H_1(T,\Z)$. If a simple closed curve $\gamma$ on $\Sphere \setminus \PP$ has two $p$-preimages, then they are homotopic by the previous proposition. Therefore, every  homotopy class of simple closed curves $\gamma$ on $\Sphere \setminus \PP$ also defines, up to sign, an element $\langle \gamma \rangle$ of $H_1(T,\Z)$. It is clear that for any $h \in H_1(T,\Z)$ there exists a homotopy class of simple closed curves $\gamma$ such that $h=n \langle \gamma \rangle$ for some $n \in \Z$.

Since $H_1(T,\Z)\cong \Z^2$, the push-forward operator $\fh_*$ is a linear operator. It is easy to see that the determinant of $\fh_*$ is equal to the degree of $\fh$, which is in turn equal to the degree of $f$. Existence of invariant multicurves for $f$ is related to the action of $\fh_*$ on $H_1(T,\Z)$. 

\begin{proposition} 
\label{prop:homology1} Suppose that a component $\gamma'$ of the $f$-preimage of a simple closed curve $\gamma$ on $\Sphere \setminus \PP$  is homotopic $\gamma$. Take a $p$-preimage $\alpha$ of $\gamma$. Then $\fh_*(\langle\alpha\rangle)=\pm d \langle\alpha\rangle$, where $d$ is the degree of $f$ restricted to  $\gamma'$.
\end{proposition}

More generally, we obtain the following.

\begin{proposition} 
\label{prop:homology} Let $\gamma$ be a simple closed curve  on $\Sphere \setminus \PP$ such that there are two points of the postcritical set $\PP$ in each complementary component of $\gamma$. If all components of the  $f$-preimage of $\gamma$ have zero intersection number  with $\gamma$ in $\Sphere \setminus \PP$, then $\fh_*(\langle\gamma\rangle)=\pm d \langle\gamma\rangle$, where $d$ is the degree of $f$ restricted to any preimage of  $\gamma$.
\end{proposition}

\subsection*{Geometrization of a Thurston map with parabolic orbifold}
As seen above, every parabolic orbifold, which is a topological 2-sphere, can be obtained by considering a quotient of  
$\R^2$ by the action of a discrete group $ G$ of Euclidean  isometries that depends only on the signature of the orbifold.
We will call $ G$ the \emph{orbifold group}. Up to equivalence, we may thus assume that a Thurston map 
$f$ with parabolic orbifold is a self-map of the $\O_f=\R^2/ G$.

\begin{theorem} 
\label{t:4PointsCase}  
Let $f$ be a Thurston map with postcritical set $P=P_f$ and no extra marked points ($Q_f=P_f$) with parabolic orbifold. 
Then $f$ is equivalent to a quotient of a real affine map by the action of the orbifold group. 
\end{theorem}

\begin{proof} Since $\O_f$ is parabolic there are three cases: $\#P$ is either $2$, $3$ or $4$. In the first two cases, the orbifold has a unique 
complex structure and $f$ is equivalent to a quotient of a complex affine map (see \cite{DH}). In the third case, the orbifold $\O_f=\OO$,
so it is the quotient of $\R^2$ by the action of 
$$G=\langle z\mapsto z+1, z\mapsto z+i, z \mapsto -1 \rangle.$$
 Note that the elements of $G$ are either translations by an integer vector or symmetries around a preimage of a marked point. 
We will denote $$S_w\cdot z=2w-z$$ the symmetry around a point  $w\in \R^2$.  
Consider a lift $F \colon \R^2 \to \R^2$ of $f$ and denote $$\Pt=\{1/2(\Z+i\Z)\}$$ the full preimage of $P$ by the projection map. 

\begin{lemma} A lift $F$ of a continuous map $f\colon \OO \to \OO$ is affine on $\Pt$.
\label{l:AffineLift}
\end{lemma}
\begin{proof}
Since $F$ is a lift of $f$, it defines a push-forward map $F_* \colon G \to G$ such that 
$$F(g\cdot z)=F_*g\cdot F(z)\text{ for any }z\in \R^2\text{ and }g\in G.$$ It is clear that $F_*$ is a homeomorphism and it sends
 translations to translations and symmetries to symmetries: $F_*S_z=S_{F(z)}$. We immediately see that 
$$F(z+w)=Aw+F(z)$$ for some integer matrix $A$ and any $w \in \Z+i\Z$. Since $$F(0)=F(S_{1/2}\cdot1)=S_{F(1/2)}\cdot F(1)=S_{F(1/2)}\cdot( A(1,0)^T+F(0))=2F(1/2)-(A(1,0)^T+F(0)),$$ we see that $$F(1/2)=F(0)+A(1/2,0)^T.$$ Similar computations for $F(1/2i)$ and $F(1/2+1/2i)$ conclude the proof of the lemma.
\end{proof}

\noindent
Thus $F(z)$ agrees with an affine map $L(z)=Az+b$ on $\Pt$, where $A$ is an integer matrix and $b \in 1/2(\Z+i\Z)$ and $F_*g=L_*g$ for all $g\in G$. Therefore the map $\tilde{\phi}=L^{-1}\circ F$ is $G$-equivariant and projects to a self-homeomorphism $\phi$ of $\O_f$ which fixes $P$. 

\begin{lemma}
\label{l:TrivialLift} Let $l(z)$ be a quotient of an affine map $L(z)=Az+b$ where $A$ is an integer matrix and $b \in 1/2(\Z+i\Z)$ by the action of $G$, and $\phi$ be an element of $\PMCG(\OO)$. If $l(z)\circ \phi$ has a lift $L'$ to $\R^2$ such that $L'(z)=Az+b$ for all points in $\Pt$, then $\phi$ is trivial.
\end{lemma}

\begin{proof}
If $l(z)\circ \phi$ and $l(z)$ have lifts that agree on $\Pt$, then $\phi$ must have a lift that is identical on $\Pt$.

The pure mapping class group $\PMCG(\OO)$ is a free group generated by Dehn twists $T_\alpha$ and $T_\beta$ around simple closed curves $\alpha$ and $\beta$ that lift to horizontal and vertical straight lines in $\R^2$ respectively. As a representative of $T_\alpha$ and $T_\beta$ we can take unique homeomorphisms on $\OO$ that are quotients of
 $$\left(\begin{array}{c} x \\ y \end{array}\right) \mapsto \left[\begin{array}{cc}
	1  & 2 \\
	0  & 1
\end{array}\right] \left(\begin{array}{c} x \\ y \end{array}\right)  \text{and}
\left(\begin{array}{c} x \\ y \end{array}\right) \mapsto \left[\begin{array}{cc}
	1  & 0 \\
	2 & 1
\end{array}\right] \left(\begin{array}{c} x \\ y \end{array}\right) 
$$ on $\R^2$ by the action of $G$. This representation of $\PMCG(\OO)$ is faithful, and therefore only the trivial element can have a lift which is identical on $\Pt$.
\end{proof}

By the previous lemma, the homeomorphism $\phi$ represents the trivial element of $\PMCG(\OO)$ and, hence, is homotopic to the identity relative to $P$. Define $l$ to be the quotient of $L$ by the action of $G$. Then the commutative diagram 
\[
\begin{diagram}
\node{R^2} \arrow{e,t}{\tilde{\phi}} \arrow{s,l}{F} \node{R^2}
\arrow{s,l}{L}
\\
\node{R^2} \arrow{e,t}{\id} \node{R^2}
\end{diagram} 
$$
\text{projects to the commutative diagram}
$$
\begin{diagram}
\node{\OO} \arrow{e,t}{\phi} \arrow{s,l}{f} \node{\OO}
\arrow{s,l}{l}
\\
\node{\OO} \arrow{e,t}{\id} \node{\OO}
\end{diagram}
\]
which realizes Thurston equivalence between $f$ and $l$.

On the other hand, suppose that $l_1$ and $l_2$ are quotients of two affine maps, which are Thurston equivalent. Then $l_1$ and $l_2$ are conjugate on $P$, hence  lifts thereof are conjugate on $\Pt$ by an affine map (in the case when $\O_f=\OO$ this follows from Lemma~\ref{l:AffineLift}; the other cases are similar) and the uniqueness part of the statement follows.
\end{proof}

\subsection*{Thurston linear transformation.} 
Let $Q$ be a finite collection of points in $S^2$.
We recall that a simple closed curve $\gamma \subset S^{2}-Q$ is \textit{essential} if it does not bound a disk, is \textit{non-peripheral} if it does not bound a punctured disk.

\begin{defn}A  \textit{multicurve} $\Gamma$ on $(S^{2},Q)$ is a set of disjoint, nonhomotopic, essential, nonperipheral simple closed curves on $S^{2}\setminus W$.
Let $(f,Q_f)$ be a Thurston map, and set $Q=Q_f$.
A multicurve $\Gamma$ on $S\setminus Q$ is \textit{f-stable} if for every curve $\gamma \in \Gamma$, each component $\alpha$ of $f^{-1}(\gamma)$ is either trivial (meaning inessential or peripheral) or homotopic rel $Q$ to an element of $\Gamma$. 

\end{defn}

To any multicurve is associated its \textit{Thurston linear transformation} $f_{\Gamma}:\mathbb{R}^{\Gamma}\rightarrow \mathbb{R}^{\Gamma}$, best described by the following transition matrix
\[
 M_{\gamma \delta}=\sum_{\alpha} \frac{1}{\text{deg}(f:\alpha \rightarrow \delta)}
\]
where the sum is taken over all the components $\alpha$ of $f^{-1}(\delta)$ which are isotopic rel $Q$ to $\gamma$.
Since this matrix has nonnegative entries, it has a leading eigenvalue $\lambda(\Gamma)$ that is real and nonnegative (by the Perron-Frobenius theorem).

We can now state Thurston's theorem:

\begin{thurstonthm} Let $f:S^{2} \rightarrow S^{2}$ be a marked Thurston map
with a hyperbolic orbifold.
 Then $f$ is Thurston equivalent to a rational function $g$ with a finite set of marked pre-periodic orbits 
if and only if $\lambda(\Gamma)<1$ for every $f$-stable multicurve $\Gamma$. The rational function $g$ is unique up to conjugation with an automorphism of $\mathbb{P}^{1}$. 
\end{thurstonthm}

The proof of Thurston's Theorem for Thurston maps without additional marked points is given in \cite{DH}, for the proof for marked maps see e.g. \cite{BGL}.

When a multicurve $\Gamma$ has a leading eigenvalue $\lambda(\Gamma)\geq 1$, we call it a \textit{Thurston obstruction} for $f$.
A Thurston obstruction $\Gamma$ is \emph{minimal} if no proper subset of $\Gamma$ is itself an obstruction. We call $\Gamma$ a \emph{simple} obstruction if no permutation of the curves in $\Gamma$ puts $M_\Gamma$ in the block form
$$ M_\Gamma = \left( 
\begin{array}{cc}
	M_{11}  & 0 \\
	M_{21}  & M_{22}
\end{array}
 \right),
$$ 
where  the leading eigenvalue of $M_{11}$ is less than $1$. If such a permutation exists, it follows that $M_{22}$ is a Thurston matrix of the corresponding sub-multicurve with the same leading eigenvalue as $M_\Gamma$. It is thus evident that every obstruction contains a simple one. 

In the original formulation in \cite{DH}, a Thurston obstruction was required to be invariant. Omitting this requirement makes the statement of the theorem weaker in one direction and stronger in the other direction. However, in \cite{sel1} is shown that if there exists a Thurston obstruction for $f$, then there also exists a simple $f$-stable obstruction.

The following is an exercise in linear algebra (c.f. \cite{sel}):

\begin{proposition}
\label{prop:positive}
  A multicurve $\Gamma$ is a simple obstruction if and only if there exists a positive  vector $v$ such that $M_\Gamma v \ge v$.
\end{proposition}

For what follows, let us make a note of a  particular kind of Thurston obstructions:
\begin{defn}
A {\it Levy cycle} is a multicurve 
$$\Gamma=\{\gamma_0,\gamma_1,\ldots,\gamma_{n-1}\}$$
such that each $\gamma_i$ has a nontrivial preimage $\gamma'_i$, where the  topological degree of $f$ restricted to $\gamma'_i$ is $1$ and $\gamma'_i$ is homotopic to $\gamma_{(i-1)\mod n}$ rel $Q$. A  Levy cycle is \textit{degenerate} if each $\gamma'_i$ bounds a disk $D_i$ such that the restriction of $f$ to $D_i$ is a homeomorphism and $f(D_i)$ is homotopic to $D_{(i+1)\mod n}$ rel $Q$.

\end{defn}
A Thurston map $f$ is called a {\it topological polynomial} if there exists a point
$w$ such that $f^{-1}(w)=\{ w\}$. The following was proved by Levy \cite{Levy}:
\begin{thm}
\label{th:levy}
If $f$ is a topological polynomial and $\Gamma$ is a Thurston obstruction for $f$, then $\Gamma$ contains a degenerate Levy cycle.
\end{thm}

\noindent
{\bf Example.} Let us give a simple yet instructive example. Let $f:\CC\to\CC$ be a postcritically finite polynomial of degree $\deg f=d\geq 2$
and let $p$ be a fixed point of $f$ which does not lie in $P_f$. Perform a topological surgery on $\hat\CC\simeq S^2$ inserting a topological disk
$D_x$ at each point $x\in \cup_{j\geq 0}f^{-j}(p)$. Modify the map $f$ accordingly to send $D_x$ to $D_{f(x)}$. Finally, select a new dynamics on $D=D_p$
so that there are at least two fixed points $a,b\in D$. The resulting topological polynomial $g$ has the same degree as $f$. Select the marked
set $$Q_g\equiv P_g\cup\{a,b\}.$$
With this choice of the marked, the Thurston map $(g,Q_g)$ is clearly obstructed -- a simple closed curve $\gamma\subset D$ which separates $a$ and $b$ from $P_g$
is a degenerate Levy cycle.

\subsection*{Thurston iteration on the Teichm{\"u}ller space}
For the basics of the Teichm{\"u}ller Theory see e.g. \cite{IT}. Let $S^2_n$ denote the two-sphere with $n$ marked points.
The moduli space $\cM(S^2_n)$ parametrizes  distinct complex structures on $S^2_n$. For $n\leq 3$ it consists of a single point.
For $n>3$, it is naturally identified with the $n-3$ dimensional complex manifold consisting of all $n$-tuples
$(z_1,\ldots, z_n)$ of points in $\hat\CC$ defined up to a M{\"o}bius transformation.
The Teichm{\"u}ller space $\cT(S^2_n)$ is the universal covering space of $\cM(S^2_n)$. 
We will use the notation $||\cdot ||_T$ for the Teichm{\"u}ller norm on $\cT(S^2_n)$.

The Teichm{\"u}ller space $\cT(S^2_n)$ can be naturally constructed as the space of equivalence classes of almost complex structures on $S^2_n$ with 
$\mu_1\equiv \mu_2$ if $\mu_1=h^*\mu_2$ where $h$ is a quasiconformal mapping of $S^2_n=\hat\CC$ isotopic to the identity relative the marked points.
Another interpretation of $\cT(S^2_n)$ is as the space of equivalence classes of quasiconformal mappings $\phi:S^2_n\to\hat\CC$
with $\phi_1\equiv \phi_2$ if and only if there exists a M{\"o}bius map $h:\hat\CC\to\hat\CC$ such that 
$h\circ \phi_1$ is isotopic to $\phi_2$ relative the marked points. The correspondence between the two viewpoints is 
standard: an almost complex structure $\mu$ on $S^2_n$ is obtained as the pullback of the standard structure $\sigma_0$ on $\hat\CC$
by $\phi$: 
$$\mu=\phi^*(\sigma_0).$$

Let $f:S^2\to S^2$ be a Thurston map of the  $2$-sphere with marked set $Q_f$. We denote $\cM_f$ and $\cT_f$ the moduli space and the Teichm{\"u}ller space respectively
of the sphere $S^2$ with marked points $Q_f$. It is straightforward to verify that the operation
defined on almost complex structures by $[\mu]\mapsto [f^*\mu]$ yields a well-defined analytic mapping
$$\sigma_f:\cT_f\to\cT_f$$
which we call the {\it Thurston pullback mapping}. It is equally easy to see that if $f$ and $g$ are two equivalent Thurston
maps then $\sigma_f$ coincides with $\sigma_g$ up to isomorphism of Teichm{\"u}ller spaces $\cT_f$ and $\cT_g$.

In terms of the description of $\cT_f$ by equivalence classes of homeomorphisms $\phi:S^2_n\to \hat\CC$, the mapping $\sigma_f$ is
defined as follows. We can pull back the almost complex structure $\mu=\phi^*\sigma_0$ by $f$ to 
$$\mu '\equiv f^*\mu=f^*\phi^*\sigma_0.$$
Using Measurable Riemann Mapping Theorem to integrate $\mu'$, we get  a mapping $\phi':S^2_n\to\hat\CC$ satisfying
$$\phi'^*\sigma_0=\mu'.$$
We now set
$$\sigma_f[\phi]=[\phi'].$$
We note that $\sigma_f$ projects to a {\em finite} cover of the moduli space:
\begin{prop}[Lemma 5.2 of \cite{DH}]
\label{finite cover}
Denote $p_f:\cT_f\to \cM_f$ the covering map.
There exists a tower $$\cT_f\overset{\tl p_f}{\longrightarrow}\widetilde \cM_f\overset{\bar p_f}{\longrightarrow}\cM_f$$ of covering spaces,
such that $\bar p_f$ is a finite cover, and a map $\tilde \sigma_f:\widetilde \cM_f\to\cM_f$, such that the diagram below commutes:
\[
\begin{CD}
\cT_f @>{\sigma_f}>> \cT_f    \\
@VV{\tl p_f}V @VV{p_f}V\\
\widetilde\cM_f@>{\tl\sigma_f}>> \cM_f
\end{CD}
\]

\end{prop}

The key starting point of the proof of Thurston Theorem is the following:
\begin{prop}[cf. \cite{DH}, Proposition 3.2.2]
A Thurston map $f$ is equivalent to a rational function if and only if $\sigma_f$ has a fixed point.
\end{prop}
\begin{proof}
Since the standard almost complex structure $\sigma_0$ on $\hat\CC$ is invariant under the pullback by a rational function,
the ``if'' direction is obvious. For the ``only if'' direction, consider a pair of homeomorphisms $\phi$ and $\phi'$
which describe the same point in the Teichm{\"u}ller space and such that
$\phi'=\sigma_f(\phi)$. The mapping 
$$f_\tau\equiv \phi\circ f\circ (\phi')^{-1}:\hat\CC\to\hat\CC$$
preserves the almost complex structure $\sigma_0$ by construction, and therefore is analytic, and hence rational.
Let $h$ be a M{\"o}bius map such that $\phi'$ is isotopic to $h\circ \phi$ relative $Q_f$. 
Then the rational mapping $f_\tau\circ h$ is Thurston
equivalent by $h$: the diagram
\[
\begin{CD}
(S^{2},Q_f) @>h^{-1}\circ \phi'>> \hat\CC\\
@VVfV @VV{f_\tau\circ h}V\\
(S^{2},Q_f) @>\phi>> \hat\CC
\end{CD}
\]
commutes up to isotopy relative $Q_f$.
\end{proof}

It is straightforward that
$$||d\sigma_f||_T\leq 1.$$
Moreover, when $f$ has a hyperbolic orbifold, there exists $k\in\NN$ such that
$$||d(\sigma_f^k)||_T<1$$
(see e.g. \cite{BGL}). It follows that:
\begin{prop}
\label{unique fixed pt}
Suppose that $f$ has a hyperbolic orbifold, and  $\sigma_f$ has a fixed point in $\cT_f$. Then the fixed point is unique,
and every $\sigma_f$-orbit converges to it.
\end{prop}

Weak contraction properties of $\sigma_f$ imply that non-existence of a fixed point means that for every compact subset
$K\Subset \cT_f$ and every starting point $[\tau_0]\in\cT_f$ there is a moment $j\in\NN$ when 
$[\sigma_f^j\tau_0]\notin K$. The next section gives a more precise explanation, due to K.~Pilgrim \cite{Pil1}.

\subsection*{Canonical obstructions}
For a general hyperbolic Riemann surface $W$ we denote
$\rho_W$, $d_W$, and $\text{length}_W$ the hyperbolic metric, distance, and length on $W$.
When we want to emphasize the dependence of the hyperbolic metric on the choice of the
complex structure $\tau$ on a surface $S$, we will write
$\rho_\tau$ for  the hyperbolic metric on $S_\tau\equiv (S,\tau)$, 
$\text{length}_\tau$ for the hyperbolic
length, and $d_\tau$ for the hyperbolic distance.
 For a non-trivial homotopy class of  closed curves $[\gamma]$ on $S$ 
we let $\ell_\tau([\gamma])$ denote the length of the unique geodesic representative of $[\gamma]$ in $S_\tau$.

The following is straightforward (see e.g. \cite{Pil1}):
\begin{prop}
Suppose there exists $\tau\in\cT_f$ such that for a non-trivial homotopy class of simple closed curves $[\gamma]$
the lengths
$$\ell_{\sigma^n_f \tau}([\gamma])\underset{n\to\infty}{\longrightarrow} 0.$$
Then the same property holds for any other starting point $\tau'\in\cT_f$. 
\end{prop}

\begin{defn}
The {\it canonical obstruction} $\Gamma_f$ of $f$ is the collection of all non-trivial homotopy classes $\gamma$ 
such that 
$$\ell_{\sigma^n_f \tau}([\gamma])\underset{n\to\infty}{\longrightarrow} 0$$
for some (equivalently, for all) $\tau\in\cT_f$.
\end{defn}

Pilgrim proved the following:
\begin{thm}[\cite{Pil1}]
\label{canonical1}
Suppose $f$ is a Thurston map with a hyperbolic orbifold. If the canonical obstruction is empty, then $f$ is 
Thurston equivalent to a rational function. If the canonical obstruction is non-empty, then it is a Thurston obstruction.
\end{thm} 

\noindent

Pilgrim further showed:
\begin{thm}[\cite{Pil1}]
\label{pil-bound}
Let $\tau_0\in\cT_f$. There exists a constant $E=E(\tau_0)$ such that for every
non-trivial simple closed curve $\gamma\notin \Gamma_f$ we have
$$\inf \ell_{\sigma^n_f\tau_0}([\gamma])>E.$$

\end{thm}

\subsection*{Pilgrim's decompositions of Thurston maps}
What follows is a very brief review; the reader is referred to K.~Pilgrim's book \cite{Pil2} for details.
We adhere to the notation of \cite{Pil2}, for ease of reference.

As a motivation, consider that for the canonical Thurston obstruction $\Gamma_c\ni\gamma$, 
there is a choice of complex structure $\tau$
for which $\ell_\tau([\gamma])$ is arbitrarily small, and remains small under pullbacks by $f$. It is thus natural
to think of the punctured sphere $S^2\setminus P_f$ as pinching along the homotopy classes $[\gamma]\in\Gamma_c$;
instead of a single sphere we then obtain a collection of spheres interchanged by $f$. 

More specifically,
let $f$ be a Thurston map, and $\Gamma=\cup\gamma_j$ an $f$-stable multicurve. Consider also a finite collection of 
disjoint closed annuli
$A_{0,j}$ which are homotopic to the respective $\gamma_j$. For each $A_{0,j}$ consider only non-trivial preimages; these form
a collection of annuli $A_{1,k}$, each of which is homotopic to one of the curves in $\Gamma$.
Pilgrim says that the pair $(f,\Gamma)$
 is in a {\it standard form} (see Figure \ref{fig-decomp})
if there exists a collection of annuli $A_{0,j}$, which we call {\it decomposition annuli}, as above such that the
following properties hold:
\begin{itemize}
\item[(a)] for each curve $\gamma_j$ the annuli $A_{1,k}$ in the same homotopy class are contained inside $A_{0,j}$;
\item[(b)] moreover, the two outermost annuli $A_{1,k}$ as above share their outer boundary curves with $A_{0,j}$.
\end{itemize}
We call the components of the complement of the decomposition annuli the {\it thick parts}.

\begin{figure}[ht]
\includegraphics[width=\textwidth]{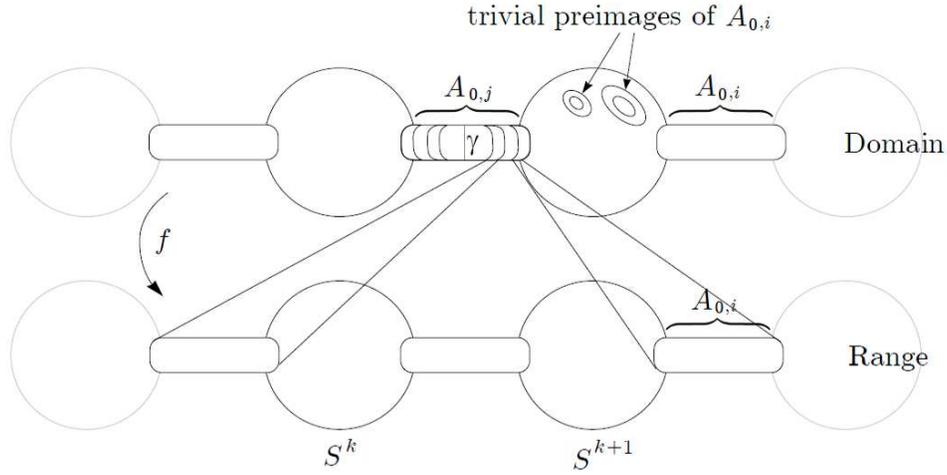}
\caption{\label{fig-decomp}Pilgrim's decomposition of a Thurston map}
\end{figure}

A Thurston map with a multicurve 
in a standard form can be decomposed as follows. First, all annuli $A_{0,j}$ are removed, leaving a collection of spheres with holes, denoted $S_0(j)$. For each $j$, there exists a unique connected component $S_1(j)$ of
$f^{-1}(\cup S_0(j))$ which has the property $\partial S_0(j)\subset \partial S_1(j)$. Any such $S_1(j)$ is a sphere with
holes, with boundary curves being of two types: boundaries of the removed annuli, or boundaries of trivial preimages of the
removed annuli. 

The holes in $S_0(j)\subset S^2$ can be filled as follows. Let $\chi$ be a boundary curve of
a component $D$ of $S^2\setminus S_0(j)$. 
Let $k\in\NN$ be the first iterate $f^k:\chi\to\chi$, if it exists. For each $0\leq i\leq k-1$ the curve 
$\chi_i\equiv f^i(\chi)$ bounds a component $D_i$ of $S^2\setminus S_0(m_i)$ for some $m_i$. Denote $d_i$ the degree
of $f:\chi_i\to \chi_{i+1}$.  Select  homeomorphisms 
$$h_i:\bar D_i\to \bar \DD\text{ so that }h_{i+1}\circ f\circ h_i^{-1}(z)=z^{d_i}.$$
Set $\tl f\equiv f$ on $\cup S_0(j)$.
Define new punctured spheres $\tl S(j)$ by adjoining cups $h_i^{-1}(\bar \DD\setminus \{0\})$
to $S_0(j)$. Extend the map $\tl f$ to each $D_i$ by setting 
$$\tl f(z)=h_{i+1}^{-1}\circ (h_i(z))^d.$$
We have thus replaced every hole with a cap with a single puncture. We call such a procedure {\it patching} a thick component.

By construction, the map 
$$\tl f:\cup \tl S(j)\to\cup\tl S(j)$$
contains a finite number of periodic cycles of  punctured spheres. For every periodic sphere $\tl S(j)$ denote
by $\cF$ the first return map $f^{k_j}:\tl S(j)\to\tl S(j).$ This is again a Thurston map.
The collection of maps $\cF$ and the combinatorial information required to glue the spheres $S_0(j)$ back together 
is what Pilgrim calls a {\it decomposition} of $f$.

Pilgrim shows:
\begin{thm}
\label{th:decomp1}
For every obstructed marked Thurston map $f$ with an obstruction $\Gamma$ there exists an equivalent map $g$ 
such that $(g,\Gamma)$ is in a standard form, and thus can be decomposed.
\end{thm}

\subsection*{Topological characterization of canonical obstructions}

The first author showed in \cite{seljmd}:
\begin{theorem}[Characterization of Canonical Obstructions] 
\label{thm:CharacterizationCanonical}
The canonical obstruction $\Gamma$ 
is a unique minimal obstruction with the following properties. 
  	\begin{itemize}
		\item If the  first-return map $F$ of a cycle of components  in $\S_\Gamma$ is a $(2,2,2,2)$-map, then  every curve of every simple Thurston obstruction for $F$ has two postcritical points of $f$ in each complementary component and the two eigenvalues of $\hat{F}_*$ are equal or non-integer.
		\item If the first-return map $F$ of   a cycle of components  in $\S_\Gamma$ is not a $(2,2,2,2)$-map or a homeomorphism, then there exists no Thurston obstruction of $F$.
	\end{itemize}
\end{theorem}

\subsection{Algorithmic preliminaries}

\subsection*{A piecewise-linear Thurston map}
For the purposes of algorithmic analysis, we will require a finite description of a branched covering $f:S^2\to S^2$.

Since we will work mainly in the piecewise linear category, it is convenient to recall here some definitions.

\paragraph{\bf Simplicial complexes} Following \cite{thu} (chapter 3.2 and 3.9) we call {\it a simplicial complex} any locally finite collection 
$\Sigma$ of simplices  satisfying the following two conditions: 
\begin{itemize}
\item a face of a simplex in $\Sigma$ is also in $\Sigma$, and 
\item the intersection of any two simplices in $\Sigma$ is either empty or a face of both. 
\end{itemize}
The union of all simplices in $\Sigma$ is the {\it polyhedron} of $\Sigma$ (written $\vert \Sigma \vert$). 

\paragraph{\bf Piecewise linear maps} A map $f:M \rightarrow N$ from a subset of an affine space into another affine space is {\it piecewise linear (PL)} if it is the restriction of a simplicial map defined on the polyhedron of some simplicial complex.

We also define {\it piecewise linear (PL) manifolds} as manifolds having an atlas where the transition maps between overlapping charts are piecewise linear homeomorphisms between open subsets of $\mathbb{R}^{n}$. It is well known that any piecewise linear manifold has a triangulation:  there is a simplicial complex $\Sigma$ together with a homeomorphism $\vert \Sigma \vert \rightarrow X$ which is assumed to be a PL map (see \cite{thu}, proof of theorem 3.10.2).

One example of such a manifold is the standard piecewise linear (PL) 2-sphere, which is nicely 
described in \cite{thu} as follows: pick any convex 3-dimensional polyhedron
 $K \subset \mathbb{R}^{3}$, and consider the charts corresponding to all the 
possible orthogonal projections of the boundary (topological) sphere $\partial K$ onto hyperplanes in $\mathbb{R}^{3}$. 
The manifold thus obtained is the {\it standard piecewise linear  2-sphere}. One can prove that another choice of 
polyhedron would lead to an isomorphic object (see exercise 3.9.5 in \cite{thu}).

It is known that in dimension three or lower, every topological manifold has a PL structure, and any two such structures are PL equivalent (in dimension 2, see \cite{Rad}, for the dimension 3 consult \cite{Bin}). 

\paragraph{\bf Piecewise linear branched covers.}
We begin by formulating the following proposition which describes how to lift a triangulation by a PL branched cover (see \cite{Doug},section 6.5.4):

\begin{prop}[{\bf Lifting a triangulation}]
\label{lift cover}
Let $B$ be a compact topological surface, $\pi:X\to B$ a finite ramified cover of $B$. Let $\Delta$ be the set of branch points of $\pi$, 
and let $\cT$ be a triangulation of $B$ such that $\Delta$ is a subset of vertices of $\cT$ ($\Delta\subset K_{0}(\cT)$ in the established notation).
Then there exists a triangulation $\cT'$ of $X$, unique up to  bijective change of indices, 
so that the branched covering map $\pi:X \rightarrow B$ sends vertices to vertices, edges to edges and faces to faces. Moreover, if $X=B$ is a standard PL
2-sphere and $\pi$ is PL, then
$\cT'$ can be produced constructively given a description of $\cT$.
\end{prop}

We consider PL maps $f$ of the standard PL 2-sphere which are postcritically finite topological branched coverings with a finite 
forward-invariant set $Q_f$ of marked points. 
We call such a map a {\it piecewise linear  Thurston map}. 

\begin{rem}
\label{representation-PL}
Note that any such covering may
be realized as a piecewise-linear branched covering map of a triangulation of $\hat \CC$ with rational vertices. An algorithmic description of a PL branched covering
could thus either be given by the combinatorial data describing the simplicial map, or  as a collection of affine maps of triangles in $\hat\CC$ 
with rational vertices. We will alternate between these descriptions as convenient. 
\end{rem}

We note:

\begin{prop}[\cite{BBY}]
\label{ThPL}
Every marked Thurston map $f$ is Thurston equivalent to a PL Thurston map.
\end{prop}

For ease of reference we state:
\begin{thm}
\label{thurston-decidable}
There exists an algorithm $\cA_1$ which, given a finite description of a marked Thurston map $f$ with  hyperbolic orbifold,
outputs $1$ if there exists a Thurston obstruction for $f$ and $0$ otherwise. In the latter case, $\cA_1$ also outputs a finite
description which uniquely identifies the rational mapping $R$ which is Thurston equivalent to $f$, and the pre-periodic orbits of
$R$ that correspond to pints in $Q_f$.
\end{thm}

The paper \cite{BBY} contains a proof of the above theorem for the case of an unmarked Thurston map ($Q_f=P_f$), the proof
extends to the general case {\it mutatis mutandis}.

\subsection*{Verifying homotopy}
Let us quote several useful results from \cite{BBY}:
\begin{prop}
\label{homotopy check}
There exists an algorithm $\cA_2$ to check whether two simple closed polygonal curves on a triangulated surface $S$ are homotopic.
\end{prop}

\begin{prop}
\label{identify-isometry}
There exists an algorithm $\cA_3$ which does the following. Given a triangulated sphere with a finite number of punctures $S=S^2-Z$ and a triangulated homeomorphism
$h:S\to S$, the algorithm identifies whether $h$ is isotopic to the identity.
\end{prop}

\subsection*{Enumeration of the multicurves and elements of the Mapping Class Group.}

We again quote \cite{BBY}:
\begin{prop}
\label{enumerate1}
Given a finite set of punctures $W$, there exist algorithms $\cA_5$, $\cA_6$ which enumerate the elements of 
$\MCG(S\setminus W)$ and $\PMCG(S\setminus W)$ respectively. 
\end{prop}
\begin{prop}
\label{enumerate2}
Given a finite set of punctures $W$, there exists an algorithm $\cA_7$ which enumerates 
 all non-peripheral multicurves on $S^{2}\setminus W$. 
\end{prop}

We combine Propositions \ref{enumerate2} and \ref{homotopy check} to formulate:

\begin{prop}
\label{enumerate3}
Given a marked PL Thurston map $f$, there exists an algorithm $\cA_8$ which enumerates all $f$-stable
multicurves. 
\end{prop}

In \cite{BBY}, \propref{identify-isometry} and \ref{enumerate1} are combined in a straightforward fashion to prove:
\begin{prop}
\label{prop-verify1}
There exists and algorithm $\cA_9$ which, given two equivalent marked PL Thurston maps $f$ and $g$ verifies the equivalence, 
by presenting an element of $\MCG(S^2\setminus Q)$ which realizes it.
\end{prop}

We also need to state a constructive version of \thmref{th:decomp1}:
\begin{prop}
\label{prop:decomp2} There exists and algorithm $\cA_{10}$ which, given an obstructed marked PL Thurston map $f$ and an obstruction $\Gamma$,
finds an equivalent PL Thurston map which is in a standard form, and such that the boundary curves of the thick parts are polygons. 
\end{prop}
\begin{proof}[Sketch of proof.] 
We use a brute force search combined with algorithm $\cA_3$ (\propref{identify-isometry}) to find a PL approximation of the map $g$ from
\thmref{th:decomp1}. We then modify the triangulation near the boundary curves of the thick parts to obtain the desired map. We leave it to
the reader to fill in the straightforward details.
\end{proof}
\subsection*{Algorithmic complexity of the Mapping Class Group}
Let us recall that a group $G$ is {\it finitely generated} if it is isomorphic to 
a quotient of the free group $F_S$ on a finite set $S$
by a normal subgroup $N\lhd F_S$. The elements of $S$ are {\it generators} of $G$. A finitely generated group is
{\it finitely presented} if there exists a finite set of words $R\subset F_S$ such that
$N$ is the normal closure of $R$ (the smallest normal subgroup of $F_S$ which contains $R$). The words in $R$ are called
{\it relators}; thus a finitely presented group can be described using a finite set of generators and relators.

The {\it Word Problem} for a finitely presented group $G$ can be stated as follows:

\medskip
\noindent
{\sl Let $S$ and $R$ be given. For a word $w$ in $F_S$ decide whether or not $w$ represents the identity in $G$.
Equivalently, for two words $w_1$, $w_2\in F_S$ decide whether $w_1$ and $w_2$ represent the same element of $G$.}

\medskip
\noindent
The {\it Conjugacy Problem} is stated similarly:

\medskip
\noindent
{\sl Let $S$ and $R$ be given. 
For two words $w_1$, $w_2$ decide whether $w_1$ and $w_2$ are conjugate elements of $G$, that is, whether there exists
$x\in G$ such that $w_1=xw_2x^{-1}$.}

\medskip
\noindent
The Word Problem is a particular case of the conjugacy problem, since being conjugate to the identity element $e\in G$ is
the same as being equal to it.

Both problems were explicitly formulated by Dehn \cite{Dehn1}, who subsequently produced an algorithm deciding the
Conjugacy Problem for a fundamental group of a closed orientable surface \cite{Dehn2}. An example of a 
finitely presented group with an algorithmically unsolvable word problem was first given in 1955 by P. Novikov \cite{novikov},
a different construction was obtained by W. Boone in 1958 \cite{boone}.

We begin by noting the following (cf. \cite{Lic,primer}):

\begin{thm}
\label{finitelypresented}
Let $S$ be an orientable surface of finite topological type. Then there exists an explicit finite presentation
of $\MCG(S)$ and of $\PMCG(S)$. This presentation can be computed from a PL presentation of $S$.
\end{thm}

As was shown by G. Hemion in 1979 \cite{hemion}:

\begin{thm}
\label{solvable-conjugacy-problem}
Let $S$ be an orientable surface of finite topological type. 
Then the  Conjugacy Problem in $\MCG(S)$ is algorithmically solvable.
\end{thm}

It is known that the Conjugacy Problem in $\MCG(S)$ is solvable in exponential time \cite{tao,hamenstadt}.

\subsection*{Hurwitz classification of branched covers}

Let $X$ and $Y$ be two finite type Riemann surfaces. We say that two finite degree branched covers $\phi$ and $\psi$ of
$Y$ by $X$ are {\it equivalent in the sense of Hurwitz} if there exist  homeomorphisms $h_0,h_1:X\to X$ such that
$$h_0\circ \phi=\psi\circ h_1.$$
An equivalence class of branched covers is known as a {\it Hurwitz class}. Enumerating all Hurwitz classes with a given
ramification data is a version of the {\it Hurwitz Problem}. The classical paper of Hurwitz \cite{hur}
gives an elegant and explicit solution of the problem for the case $X=\hat\CC$.

We will need the following narrow consequence  of Hurwitz's work  (for a modern treatment, see \cite{barth}:
\begin{thm}
\label{th:hurwitz}
There exists an algorithm $\cA$ which, given PL branched covers $\phi$ and $\psi$ of PL spheres and a PL 
homeomorphism $h_0$ mapping the
critical values of $\phi$ to those of $\psi$, does the following:
\begin{enumerate}
\item decides whether $\phi$ and $\psi$ belong to the same Hurwitz class or not;
\item if the answer to (1) is affirmative, decides whether there exists a homeomorphism $h_1$ such that $h_0\circ \phi=\psi\circ h_1.$
\end{enumerate}
\end{thm}

\section{Classification of marked Thurston maps with parabolic orbifolds}

Let $f$ be a Thurston map with postcritical set $P_f$ and marked set $Q_f\supset P_f$.
In what follows, we will drop the subscript $f$ and will denote these sets simply $P$ and $Q$.
 Let $\Gamma$ be a Thurston obstruction for $f$.
The goal of this section is to prove the following theorem:

\begin{theorem}
\label{th:degenerateLevy}
  Let $f$ be a Thurston map with postcritical set $P$ and marked set $Q\supset P$ such that the associated orbifold is parabolic and the associated matrix is hyperbolic. Then either $f$ is equivalent to a quotient of an affine map or $f$ admits a degenerate Levy cycle.

Furthermore, in the former case the affine map is defined uniquely up to a conjugacy.
\end{theorem}

\begin{remark} We note that in the case when the associated matrix has eigenvalue $\pm1$, the two options are not mutually exclusive. 
\end{remark}


\subsection{The case when  the associated matrix is expanding}
We will first derive \thmref{th:degenerateLevy}  in the case when the matrix of the corresponding affine map is expanding.

\begin{theorem}   Let $(f,Q)$ be a Thurston map with postcritical set $P$ and marked set $Q\supset P$ with parabolic orbifold, such that $(f,P)$ is equivalent to a quotient $l$ of a real affine map $L(z)=Az+b$ by the orbifold group where both eigenvalues of $A$ have absolute value greater than 1.  Then $(f,Q)$ is equivalent to a quotient of a real affine map by the action of the orbifold group if and only if $f$ admits no degenerate Levy cycle.
\label{t:ExpandingCase}
\end{theorem}

\begin{proof} Let $\phi_0$ and $\phi_1$ realize Thurston equivalence between $(f,P)$ and $l$, i.e. $\phi_0\circ f = l \circ \phi_1$ and $\phi_0$ is homotopic  to $\phi_1$ relative  $P$. The following argument is fairly standard (compare \cite{Shi,BonkMeyer}). We can lift the homotopy between $\phi_0$ and $\phi_1$ by $l$ to obtain a homotopy between $\phi_1$ and the lift $\phi_2$ of $\phi_1$ (see the commutative diagram below).

$$
\begin{diagram}
\node{\ldots}  \arrow{s,l}{f} \node{\ldots}
\arrow{s,l}{l}
\\
\node{\O_f} \arrow{e,t}{\phi_2} \arrow{s,l}{f} \node{\O_f}
\arrow{s,l}{l}
\\
\node{\O_f} \arrow{e,t}{\phi_1} \arrow{s,l}{f} \node{\O_f}
\arrow{s,l}{l}
\\
\node{\O_f} \arrow{e,t}{\phi_0} \node{\O_f}
\end{diagram}
$$

Since both eigenvalues of $A$ have absolute value greater than 1, the map $l$ is expanding with respect to the Euclidean metric on $\O_f$. This implies that the distance between $\phi_n(z)$ and $\phi_{n+1}(z)$ decreases at a uniform geometric rate for all $z \in \O_f$. We see that the sequence $\{\phi_n\}$ converges uniformly to a semi-conjugacy $\phi_\infty$ between $f$ and $l$. One of the following is true then.

\textbf{Case I.} Suppose $\phi_\infty$ is injective on $Q$. Let $n$ be such that $d(\phi_\infty(z),\phi_n(z))<\eps$ for all $z \in \O_f$, where $\eps$ is small. Consider a homotopy, which is nontrivial only in the $\eps$-neighborhood of $Q\sm P$ that transforms $\phi_n$ to $\phi'_n$ such that  $\phi'_n$ agrees with $\phi_\infty$ on $Q$. Then the lift $\phi_{n+1}'$ of $\phi_n'$ is $2\eps$-close to $\phi_\infty$ and, hence, also agrees with $\phi_\infty$ on $Q$ if $\eps$ was chosen small enough.
It is also clear that for $\eps$ small $\phi'_{n}$ and $\phi'_{n+1}$ are homotopic relative $Q$, realizing Thurston equivalence between $(f,Q)$ and $(l,\phi_\infty(Q))$.

\textbf{Case II.} Suppose $\phi_\infty$ is not injective on $Q$. Consider a point $z_0$ which is the image of at least two different points $q_1$ and $q_2$ in $Q$; obviously $z_0$ is either periodic or pre-periodic. First we show that $f(q_1)\neq f(q_2)$. Indeed, if $f(q_1)= f(q_2)$, then the distance between $\phi_n(q_1)$ and $\phi_n(q_2)$ is uniformly bounded from below by the minimum distance between any two points in the same fiber of $l$, which contradicts the fact that $\phi_\infty(q_1)=\phi_\infty(q_2)$. Therefore $l(z_0)$ is also 
the image of at least two different points in $Q$ and so on. Thus, we can assume that $z_0$ is periodic with period, say, $m$.

Consider a small simple closed curve $\gamma$ around $z_0$ (for example, we can take a circle around $z_0$ of radius $\eps$). Since $z_0$ is periodic it is not a critical point of $l$; the $m$-th iterate of $l$ sends $\gamma$ to another simple closed curve $\gamma'$ around $z_0$, which is evidently homotopic to $\gamma$ relative  $\phi_\infty(Q)$, in one-to-one fashion, moreover the disk bounded by $\gamma$  that contains $z_0$ is mapped homeomorphically to the disk bounded by $\gamma'$. This yields that, for $n$ large enough, $\alpha'=\phi_n^{-1}(\gamma')$ and $\alpha=\phi_{n+m}^{-1}(\gamma)$  are homotopic relative $Q$ and $f^m$ homeomorphically maps a disk bounded by $\alpha$ to a disk bounded by $\alpha'$. We see that $\alpha, f(\alpha), \ldots, f^{m-1}(\alpha)$ form a degenerate Levy cycle.



\end{proof}

\begin{rem}
Note that if $P$ has only 3 points, the matrix $A$ is a multiplication by a complex number and both eigenvalues of $A$ have the same absolute value, which is greater than 1. 
\end{rem}

\subsection{When the associated matrix is hyperbolic}
We now want to prove \thmref{th:degenerateLevy}  for any $(2,2,2,2)$-map such that the corresponding linear transformation is hyperbolic but not  expanding.
Throughout this section we assume that $(f,Q)$ is a Thurston $(2,2,2,2)$-map with postcritical set $P$ and marked set $Q\supset P$, such that $(f,P)$ is equivalent to a quotient $l$ of a real affine map $L(z)=Az+b$ by the orbifold group where both eigenvalues of $A$ are not equal to $\pm1$.

\begin{definition} Let $f$ be a  $(2,2,2,2)$-map and let $z$ be an $f$-periodic point with period $n$. Fix a universal cover $F$ of $f$ and take a point  $\tilde{z}$  in the fiber of $z$. If $z \notin P$, we define the \emph{Nielsen index} $\text{ind}_{F,n}(\tilde{z})$ to be the unique element $g$  of the orbifold group $G$ such that $F^n(\tilde{z})=g\cdot \tilde{z}$. If $z \in P$ then the Nielsen index of $z$  is defined up to pre-composition with the symmetry around $z$.
\end{definition}

Below, when we say that a point $z$ has a period $n$ we do not imply that $n$ is the minimal period of $z$.

\begin{definition}  Let $f$ be a  $(2,2,2,2)$-map  and let $z_1, z_2$ be $f$-periodic points with period $n$. We say that $z_1$ and $z_2$ are in the same \emph{Nielsen class of period} $n$ if there exists a universal cover $F_n$ of $f^n$ and points  $\tilde{z}_1,\tilde{z}_2$  in the fibers of $z_1,z_2$ respectively, such that both $\tilde{z}_1$ and $\tilde{z}_2$ are fixed by $F_n$. We say that  $z_1$ and $z_2$ are in the same \emph{Nielsen class} if there exists an integer $n$ such that they are in the same class of period $n$.
\end{definition}

Note that if two points are in the same Nielsen class of period $n$, then they are in the same Nielsen class of period $mn$ for any $m\ge1$. Clearly, being in the same Nielsen class (without specifying a period) is an equivalence relation, which is preserved under Thurston equivalence.

\begin{lemma} Periodic points $z_1$ and $z_2$ of period $n$ are in the same Nielsen class if and only if, for any universal cover $F$ of $f$, there exist points $\tilde{z}_1,\tilde{z}_2$  in the fibers of $z_1,z_2$ respectively such that $\text{ind}_{F,n}(\tilde{z}_1)=\text{ind}_{F,n}(\tilde{z}_2)$.
\label{l:NielsenIndex}
\end{lemma}
\begin{proof}

If $g=\text{ind}_{F,n}(\tilde{z}_1)=\text{ind}_{F,n}(\tilde{z}_2) \in G$ for some universal cover $F$ of $f$ and points $\tilde{z}_1,\tilde{z}_2$ , then $g^{-1}\cdot F^n(\tilde{z}_i)=\tilde{z}_i$ for $i=1,2$ and hence $z_1$ and $z_2$ are in the same Nielsen class.

In the other direction, suppose  $F_n(\tilde{z}_i)=\tilde{z}_i$ for $i=1,2$ and some cover $F_n$ of $f^n$. For any cover $F$ of $f$, its iterate $F^n$ can be written in the form $F^n=g \cdot F_n$ where $g \in G$. Therefore $\text{ind}_{F,n}(\tilde{z}_1)=\text{ind}_{F,n}(\tilde{z}_2)=g$.
\end{proof}

The following statement is obvious.

\begin{lemma}
\label{l:iterates} A Thurston map $f$ admits a (degenerate) Levy cycle if and only if so does its iterate $f^n$. Two points $z_1, z_2$ are in the same Nielsen class for $f$ with period $m$ if and only if they are in the same Nielsen class for $f^n$ with period $m/\gcd(m,n)$.
\end{lemma}

\begin{lemma} Let $A$ be a $2\times2$ integer matrix with determinant greater than 1 and both eigenvalues not equal to $\pm1$. If $v$ is a non-zero integer vector, then $A^{-n} \cdot v$ is non-integer for some $n>0$.
\label{l:da}
\end{lemma}
\begin{proof} Suppose, on contrary, that $A^{-n} \cdot v=(p_n,q_n)^T$ where $p_n, q_n \in \Z$ for all $n>0$. If both eigenvalues of $A$ have absolute values greater than 1, then evidently both $p_n$ and $q_n$ tend to $0$. Thus for some $n$, $p_n=q_n=0$ and, multiplying $(p_n,q_n)^T$ by $A^n$, we see that $v$ is also a zero vector. Since by assumption, eigenvalues are not equal to $\pm1$, the only case we need to consider is when $A$ has two distinct real  irrational eigenvalues $|\lambda_1|>1$ and $|\lambda_2|<1$.

In this case, $A$ is diagonalizable; write $v$ as a linear combination $v=v_1+v_2$ of two eigenvectors $v_1=(x_1,y_1)^T$ and $v_2=(x_2,y_2)^T$. Then $A^{-n} \cdot v = \lambda_1^{-n}v_1 + \lambda_2^{-n}v_2$ so $p_n= x_1\lambda_1^{-n}+x_2\lambda_2^{-n}$ and $q_n= y_1\lambda_1^{-n}+y_2\lambda_2^{-n}$.
Note that $q_n \lambda_1^{-n}= y_1\lambda_1^{-2n}+y_2(\lambda_1\lambda_2)^{-n}\to 0$ because $|\lambda_1|>1$ and $\lambda_1\lambda_2=\det A>1$; thus $|\lambda_1|^{-n} =o(1/|q_n|)$.
 Then
$$
\left| \frac{p_n}{q_n} - \frac{x_2}{y_2} \right| = 
\left| \frac{(x_1y_2-x_2y_1)\lambda_1^{-n}}{y_2q_n} \right| = O\left( \left| \frac{\lambda_1^{-n}}{q_n} \right|\right)
=o\left(  \frac{1}{q_n^2} \right).
$$

Since $\lambda_2$ is a quadratic algebraic number, the ratio $x_2/y_2$ must also be quadratic algebraic. No quadratic algebraic number, however, can be approximated by rationals this way and we arrive at a contradiction.
\end{proof}

\begin{corollary}
\label{c:da} Let $L(z)=Az+b$ be a real affine map such that $A$ is an integer matrix with $|\det A|>1$  and $b$ is a vector with entries in $\frac{1}{q}\Z$ for some $q\in\N$, and assume that $A$ has eigenvalues not equal to $\pm 1$. If $L^{-n}(v) \in  \frac{1}{q}\Z$  for all $n\ge0$, then $v$ is equal to the fixed point of $L$.
\end{corollary}
\begin{proof}
The case when $b=0$ follows immediately from the previous lemma. If $b \neq 0$, we conjugate $L(z)$ by $t(z)=z-x$, where $x$ is the unique fixed point of $L(z)$, to obtain a real linear map $L'(z)$. Then $L'(z)$ and $t(v)$ also satisfy the assumption of this corollary (possibly with a different $q$) and we conclude that $t(v)=v-x=0$.
\end{proof}

\begin{definition} Suppose that one of the complementary components to a simple closed curve $\gamma$ in $(\OO,Q)$ contains at most one point of $P$ (so that $\gamma$ is trivial in $(\OO,P)$). We call that component $\text{int}(\gamma)$ the \emph{interior} of $\gamma$.
\end{definition}

\begin{proposition}  Let $\{\gamma_n\}$ be a sequence of simple closed curves in $(\OO,Q)$ that are inessential in $(\OO,P)$ such that a $(2,2,2,2)$-map 
$f$ sends $\gamma_{n+1}$ to $\gamma_n$ and $Q'=\text{int}(\gamma_n)\cap Q$ is the same for all $n$. Then there exits $m$ such that all points in $Q'$ are periodic with period $m$ and lie in the same Nielsen class.
\label{p:InfinitePullback}
\end{proposition}

\begin{proof} Since all $\gamma_n$ are inessential in $(\OO,P)$, the map $f$ sends $\Int(\gamma_{n+1})$ homeomorphically onto $\Int(\gamma_{n})$. Indeed, $\Int(\gamma_{n})$ contains at most one critical value of $f$, and if it does contain a critical value $p$, then the unique preimage of $p$ in $\Int(\gamma_{n+1})$ must be $p$ itself, which is not a critical point. Therefore $f$ is a bijection on $Q'$ and  every point in $Q'$ is periodic; denote $m$ the least common multiple of the periods of points in $Q'$. It is enough to prove that for $f^m$, the subset $Q'$ of the set of fixed points  lies in a single Nielsen class.

Let $F$ be a universal cover of $f^m$ such that a point $\tilde{r}$ in the fiber of $r \in Q'$ is fixed by $F$. Let $s$ be any other point in $Q'$. Since $Q' \cap P$ contains at most one point, we may assume that $s \notin P$. Connect $r$ and $s$ by a curve $\alpha_1$ in $\Int{(\gamma_1)}\sm P$. The lift $\tilde{\alpha_1}$ of $\alpha_1$ that starts at $\tilde{r}$ will end at some point $\tilde{s}_1$ in the fiber of $s$. Denote $g_1=\text{ind}_{F,1}(\tilde{s}_1)$; in other words, $g$ is a unique transformation in  $G$ such that $F(\tilde{s}_1)=g_1\cdot\tilde{s}_1$. Consider the lift $\alpha_2$ of $\alpha_1$ by $f^m$ that starts at $r$.  Since $r \in \Int(\gamma_{m+1})$ the whole curve $\alpha_2$ lies in $\Int(\gamma_{m+1})$ and, thus, ends in the unique preimage of $s$ within $\Int(\gamma_{m+1})$, which is $s$ itself. Therefore the  lift $\tilde{\alpha_2}=F^{-1}(\alpha_1)$ of $\alpha_2$ that starts at $\tilde{r}$ will end at some point $\tilde{s}_2$ in the fiber of $s$. We conclude by induction that $$F^{-n}(\tilde{s}_1)=\tilde{s}_{n+1}$$ with $\tilde{s}_{n+1}$ in the fiber of $s$ for all $n$. Denote $$g_n=\text{ind}_{F,1}(\tilde{s}_n),$$ which is a unique element of $G$ such that $$\tilde{s}_{n-1}=F(\tilde{s}_n)=g_n\cdot\tilde{s}_{n}\text{ for all }n\ge2.$$ Then $$g_n\cdot\tilde{s}_{n}=\tilde{s}_{n-1} = F(\tilde{s}_n)=F(g_{n+1} \cdot \tilde{s}_{n+1})=F_*(g_{n+1}) \cdot F(\tilde{s}_{n+1})= F_*(g_{n+1}) \cdot \tilde{s}_{n}.$$  Since $\tilde{s_n} \notin \Pt$, this yields $F_*(g_{n+1})=g_n$ for all $n$.

 By Theorem~\ref{t:4PointsCase} $(f^m,P)$ is Thurston equivalent to a quotient of an affine map $L(z)=Az+b$; the push-forward map $F_*$ is easily computed: for a translation $$T_v \cdot z = z+v$$ we get  $$F_*(T_v)=T_{Av}$$ and for a symmetry $$S_v \cdot z=2v-z$$ we get $$F_*(S_v)=S_{Av+b}.$$ It follows that if  $g_1$ is  equal to a translation $T_v$, then   $$g_n=T_{A^{-n+1}v}\text{ for all }n\geq2.$$ In particular, all ${A^{-n+1}v}$ are integer vectors and Corollary~\ref{c:da} yields $v=0$. We see that $$\text{ind}_{F,1}(\tilde{s}_1)=g_1=\id$$ and $\tilde{s}_1$ is fixed by $F$ so $r$ and $s$ are in the same Nielsen class.

Similarly, if $g_1=S_v$, then $g_n=S_{L^{-n+1}(v)}$ and  Corollary~\ref{c:da} also applies, implying that $L(v)=v$ and $g_1=g_n$ for all $n$. But then $$\tilde{s}_1=g_2\cdot g_3 \cdot \tilde{s}_3=S^2_v \tilde{s}_3=\tilde{s}_3$$ and $F^2$ fixes both $\tilde{r}$ and $\tilde{s_1}$.
\end{proof}

\begin{proposition} A map $f$ admits a degenerate Levy cycle if and only if there exist two distinct periodic points in $Q$ in the same Nielsen class.
\label{p:NielsenLevy}
\end{proposition}

\begin{proof} In view of Lemma~\ref{l:iterates}, we can freely replace $f$ by any iterate of $f$ and assume that all periodic points in $Q$ are fixed. Suppose an essential simple closed curve $\gamma$ forms  a Levy cycle of length 1, i.e. $f(\gamma)$ is homotopic to $\gamma$ and the degree of $f$ restricted to $\gamma$ is 1. Recall that $(f,P)$ is equivalent to a quotient of $z \mapsto Az+b$. Essential simple closed curves on $(\OO,P)$ are in one-to-one correspondence with non-zero integer vectors $(p,q)^T$ such that $q\ge0$ and $\gcd(p,q)=1$. The action of $f$ on the first homology group  of $(\OO,P)$ is (in the appropriate basis) the multiplication by $A$ so that if $\gamma_1$ and $\gamma_2$ are simple closed essential curves labeled by $(p_1,q_1)$ and $(p_2,q_2)$ such that $f(\gamma_2)=\gamma_1$, then $$A(p_2,q_2)^T=\pm d(p_1,q_1)^T$$ where $d$ is the degree of $f$ restricted to $\gamma_2$. This yields that $\gamma$ is inessential in $(\OO,P)$ because otherwise $A$ must have an eigenvalue $\pm 1$, which contradicts the assumptions. We are now in the setting of Proposition~\ref{p:InfinitePullback} and we see that all points in $\Int(\gamma)$ are in the same Nielsen class with some period $m$.

Suppose now that there are at least two fixed points of $f$ in the same Nielsen class $C$. Consider all points of $Q$ in this class. Replacing $f$ by an iterate, we may assume that all of them are in the same Nielsen class of period 1, i.e. there exists a universal cover $F$ of $f$ such that for each point $q \in C$, some lift $\tilde{q}$ is fixed by $F$. Note that these will be the only fixed points of $F$. Note that $C$ contains at most 1 point of $P$. 

\begin{lemma}
\label{l:homotopy}
There exists a simple closed curve $\gamma$ on $(\OO,\Q)$, which is inessential in $(\OO,P)$, such that $\Int(\gamma) \cap Q = C$ and some lift $\tilde{\gamma}$ of $\gamma$ separates $F$-fixed lifts of points in $C$ from the rest of the lifts of points in $Q$.
\end{lemma}
\begin{proof} Let $\tilde{s}$ be a fixed lift of  a point $s\in C$ and assume $s \notin P$. Let $\tilde{s}'$ be any other point in $\R^2 \sm \tilde{Q}$. Take a path connecting $\tilde{s}$ and $\tilde{s}'$ in $\R^2 \sm \tilde{Q}$ and construct a $G$-equivariant homotopy $H_t(z)$ that moves $\tilde{s}$ along this path. Since $H$ is $G$-equivariant and the chosen path is disjoint from $\Pt$, the former projects to a homotopy $h$ on $(\OO,P)$. Thus we find a map $f'=h_1 \circ f\circ h_1^{-1}$, which is conjugate to $f$, such that its lift $$F'=H_1 \circ F \circ H_1^{-1}$$ fixes $\tilde{s}'$ instead of $\tilde{s}$. Continuing in this manner, we can move all fixed points of $F$ into a small round disk $D$, which contains no other lifts of points in $Q$. If $C \cap P =\emptyset$, then  $D$ projects homeomorphically to $\OO$; if $p=C\cap P$, then we can take $p$ to be in the center of $D$. In both cases, the boundary of this disk and the projection thereof satisfy the conclusion of the proposition for the modified $f$; the images of these curves by the conjugating maps will do the same for the map $f$ itself.
\end{proof}

Consider a curve $\gamma_0$ as in the lemma above. Since $\tilde{\gamma}_0$ surrounds all fixed points of $F$, so does its preimage $\tilde{\gamma_n}=F^{-n}(\tilde{\gamma}_0)$, which projects to a simple closed curve $\gamma_n$ on $(\OO,Q)$. Let $a_n$ be the intersection number of $\gamma_{n}$ and $\gamma_{n+1}$ (we may always assume that $\gamma_0$ and $\gamma_1$ have only finitely many intersections, all of which are transversal). Clearly $a_n$ is non-increasing. If $a_n=0$ for some $n$, then $\gamma_n$ and $\gamma_{n+1}$ are disjoint and have the same marked points in their interiors, hence they are homotopic and $\gamma_n$ forms a Levy cycle of length 1. Otherwise, by truncating the sequence, we may assume that $a_n=a>0$ for all $n\ge0$. In this case, $\gamma_n \cup \gamma_{n+1}$ is mapped homeomorphically to $\gamma_{n-1}\cup \gamma_n$. Let $\beta_0 \subset \gamma_{0}\cup \gamma_1$ be a simple closed curve and denote $\beta_n$ to be a unique one-to-one $f^n$-preimage of $\beta_0$ that is a subset of $\gamma_n \cup \gamma_{n+1}$.  

\begin{lemma}
\label{l:inessential}
Let $\{\beta_n\}$ be a sequence of simple closed curves in $(\OO,Q)$ such that 
$f$ sends $\beta_{n+1}$ to $\beta_n$ with degree 1. Then all $\beta_n$ are inessential in $(\OO,P)$.
\end{lemma}

\begin{proof}
 Since the degree of $f$ restricted to any $\beta_n$ is 1, the following holds: $$A^{n}(p_n,q_n)^T=\pm(p_0,q_0)^T,$$ where $\beta_n$ corresponds to $\pm(p_n,q_n)^T$ in the first homology group of $(\OO,P)$. By Lemma~\ref{l:da}, we see that $p_1=q_1=0$. 
\end{proof}

Thus, all $\beta_n$ are inessential in $(\OO,P)$. As $\beta_0$ was any simple closed curve in $\gamma_{0}\cup \gamma_1$, we infer that there exists a  connected component $M$ of $\OO \sm (\gamma_{0}\cup \gamma_1)$ that contains at least 3 points of $P$. Indeed, if there exists a component with exactly 2 points of $P$, then the boundary $\beta$ thereof is essential in $(\OO,P)$, which is a contradiction. If there are exactly 4 components, each containing a single point of $P$, one can find a simple closed curve in $\gamma_{0}\cup \gamma_1$ that has exactly 2 points in each complementary component by induction on the number of components. Indeed, it is easy to see that there always exists a pair of adjacent components such that their closures intersect at exactly 1 boundary arc; removing that arc reduces the number of components by 1. From now on we assume that $\beta_0=\partial M$.

Denote $$Q_n=\Int(\beta_n) \cap Q.$$ Since $\# Q_n$ is non-increasing, we may assume, by further truncating the sequence if necessary, that $\# Q_n$ is constant. Recall that we assumed that all points in $Q$ are either fixed or strictly pre-periodic. This implies $Q_n=Q_0$ for all $n$. We are now in the setting of Proposition~\ref{p:InfinitePullback}, which yields all points in  $Q_0$ are in the same Nielsen class.  Recall that $$\Int(\gamma_0) \cap Q =\Int(\gamma_1) \cap Q = C$$ contains at most 1 point of $P$ so M is in the complement of $\Int(\gamma_0)\cup \Int(\gamma_1)$.  We see that all marked points in the complement of $M$ are in $C$. This readily implies that all $\gamma_n$ are homotopic to $\beta$ and $\gamma_0$ forms a Levy cycle of length 1,
which concludes our proof of \propref{p:NielsenLevy}.
\end{proof}

\begin{proposition} Let $\{\gamma_n\}$ be a sequence of essential simple closed curves in $(\OO,Q)$ such that $f$ sends $\gamma_{n+1}$ to $\gamma_n$ with degree 1. Then $f$ admits a degenerate Levy cycle.
\label{p:InfiniteLevy}
\end{proposition}

\begin{proof} By Lemma~\ref{l:inessential} all $\gamma_n$ are inessential in $(\OO,P)$. Replacing $\{\gamma_n\}$ by a subsequence $\{\gamma_{nk+l}\}$, for some integers $k,l$, we can always assume that  $Q'=\Int(\gamma_n) \cap Q$ is the same for all $n$ (see the previous proof). Since $\gamma_n$ are essential in $(\OO,Q)$, the set $Q'$ contains at least two points. By Proposition~\ref{p:InfinitePullback}, these two points are in the same Nielsen class and Proposition~\ref{p:NielsenLevy} implies existence of a Levy cycle.
\end{proof}

\begin{corollary}
\label{c:FiniteDepth}
If $f$ admits no Levy cycle, then for every simple closed curve $\gamma$ in $(\OO,Q)$, which is inessential in $(\OO,P)$, there exists an integer $d$ such that all connected components of $f^{-d}(\gamma)$ are inessential in $(\OO,Q)$.
\end{corollary}
\begin{proof} Define the \emph{depth} of $\gamma$ to be the largest integer $d(\gamma)$ such that $f^{-d(\gamma)}(\gamma)$ has an essential component. The goal is to prove that $d_\gamma$ is finite for all inessential in $(\OO,P)$ curves. Clearly, $$d(\alpha)=1+\max d(\alpha_i)$$ where $\alpha_i$ are the connected components of the preimage of a simple closed curve $\alpha$. Therefore, if $\gamma$ has infinite depth, so does at least one of its preimages $\gamma_1$. We construct thus an infinite sequence of essential in $(\OO,Q)$ curves $\gamma_n$ such that $f$ maps $\gamma_{n+1}$ to $\gamma_n$. Since a preimage of a trivial in $(\OO,P)$ curve is also trivial in $(\OO,P)$, truncating the sequence if necessary, we may assume that all $\gamma_n$ are either all trivial or all non-trivial in $(\OO,P)$. In both cases, the degree of $f$ restricted to $\gamma_n$ is 1 for all $n$ and the previous proposition yields existence of a Levy cycle.
\end{proof}

The above result immediately implies:
\begin{corollary}
\label{c:noobs}
If $f$ admits no Levy cycle, then every curve of every simple Thurston obstruction for $f$ is essential in $(\OO,P)$.
\end{corollary}
For future reference, let us summarize:

\begin{corollary}
\label{c:generalobs}
\begin{itemize}
\item
  Let $f$ be  marked $(2,2,2,2)$-map such that the corresponding matrix does not have eigenvalues $\pm 1$.
Then $f$ is equivalent to a quotient of an affine map with marked pre-periodic orbits if and only if every curve  
of every simple Thurston obstruction for $f$ has two postcritical points of $f$ in each complimentary component.
\item A marked  Thurston map $f$ with a parabolic orbifold that is not  $(2,2,2,2)$ is equivalent to a quotient of an 
affine map if and only it admits no Thurston obstruction.  
\end{itemize}  
\end{corollary}

\begin{proof}
  The first statement follows immediately from the previous corollary and Theorem~\ref{th:geometrization1}. The second statement follows from Theorem~\ref{t:ExpandingCase}. Indeed, suppose that a Thurston map $f$ with a parabolic orbifold with signature other than $(2,2,2,2)$ admits a simple obstruction $\Gamma$. If the signature is $(\infty,\infty)$ or $(\infty,2,2)$, then $f$ is an obstructed  topological polynomial and therefore admits a Levy cycle (\thmref{th:levy}) which is necessarily degenerate. In other cases all points in the postcritical set of $f$ are not critical. As before, we pass to an iterate of $f$ such that all marked points are either fixed or pre-fixed (in particular, all postcritical points are fixed in this case) and set the interior  $\Int(\gamma)$ to be  the unique component of the complement to $\gamma$ which contains at most 1 postcritical point.  
  Up to passing to yet another iterate of $f$, we may assume that some $\gamma \in \Gamma$ has a preimage $\gamma'$ homotopic to $\gamma$. If $\Int(\gamma)$ contains no postcritical points, then $\Int(\gamma')$ contains no critical points. If $\Int(\gamma)$ contains a postcritical point $p$, then $p$ is the unique preimage of itself within $\Int(\gamma')$, and again $\Int(\gamma')$ contains no critical points. Therefore in both cases $\{\gamma\}$ is a degenerate Levy cycle.
\end{proof}

\begin{defn}
Denote by $\RMCG(\OO,Q)$ the \emph{relative mapping class group} of $(\OO,Q)$, which is the group of all mapping classes $\phi$ for which there exists a universal cover $\tilde{\phi}$ that is identical on $\tilde{Q}$. We now need the following generalization of Lemma~\ref{l:TrivialLift}.
\end{defn}

\begin{theorem} The group $\RMCG(\OO,Q)$ is generated by Dehn twists around trivial curves in $(\OO,P)$ and by second powers of Dehn twists around non-trivial inessential curves in $(\OO,P)$.
\label{t:rmcg}
\end{theorem}

\begin{proof} The proof of this theorem is similar to the proof of classical results on generators of $\MCG$ (cf. \cite{primer}). We proceed by induction on the number of points in $Q$. When $Q=P$, the statement reduces to Lemma~\ref{l:TrivialLift}. 

Suppose that the statement is true for the marked set $Q\subset \OO$ and let us prove it for $Q'=Q \cup \{q\}$ where $q \notin Q$. There exists an obvious projection map $$\Forget:\PMCG(\OO,Q')\to\PMCG(\OO,Q),$$ which simply regards a self-homeomorphism of $(\OO,Q')$ as a self-homeomorphism of $(\OO,Q)$, forgetting about the existence of $q$. Take any $\phi \in \RMCG(\OO,Q')$; by inductive assumption, $\Forget(\phi) \in \RMCG(\OO,Q)$ can be represented in $\PMCG(\OO,Q)$ as  a product $\prod  T^{n_i}_{\gamma_i}$ of Dehn twists around trivial curves in $(\OO,P)$ and second powers of Dehn twists around non-trivial inessential curves in $(\OO,P)$. 
We may assume that every $\gamma_i$ does not pass through the point $q$; otherwise we replace $\gamma_i$ by a curve $\gamma_i'$, which is homotopic to $\gamma_i$ relative $Q$, that does not pass through $q$ (note that in this case the homotopy class of $\gamma'_i$ in $(\OO, Q)$  is not uniquely defined). Then $$\Forget(T_{\gamma_i})=T_{\gamma_i}$$ where $T_{\gamma_i}$ is viewed as an element of both  $\PMCG(\OO,Q')$ and  $\PMCG(\OO,Q)$. Thus $$\psi = \phi \circ \left( \prod  T^{n_i}_{\gamma_i} \right)^{-1}$$ is a well defined element of $\PMCG(\OO,Q')$ such that 
$\Forget(\psi)=\id$. It is, hence, sufficient to show that every $\psi \in \RMCG(\OO,Q')$ such that $\Forget(\psi)=\id$ is generated by (squares of) Dehn twists.

Recall the Birman exact sequence (cf. \cite{primer}):
$$  1 \longrightarrow \pi_1(\OO\sm Q,q)  \stackrel{\Push}{\longrightarrow} \PMCG(\OO,Q') \stackrel{\Forget}{\longrightarrow} \PMCG(\OO,Q) \longrightarrow 1,$$
where   $\Push$ is the map that sends a loop based at $q$ to a homeomorphism, which can be obtained at the end of a homotopy relative $Q$ that  pushes the point $q$ along this loop. Since $\psi$ lies in the kernel of $\Forget$, we infer $\psi=\Push(\gamma)$ for some loop $\gamma \in \pi_1(\OO\sm Q,q)$. Since $\psi$ is also an element of $\RMCG(\OO,Q')$, it has a universal cover $\tilde{\psi}$ which is identical on the fiber of $q$. Pick a point $\tilde{q}$ in this fiber; in particular, $\tilde{q}$ is fixed by $\tilde{\psi}$. It is clear that the lift $\tilde{\gamma}$ of $\gamma$ starting at $\tilde{q}$ ends at $\tilde{\psi}(\tilde{q})=\tilde{q}$, i.e. $\tilde{\gamma}$ is a loop based at $\tilde{q}$.  On the other hand, each loop $\tilde{\gamma}$ in $\R^2 \sm \tilde{Q}$ based at $\tilde{q}$ produces a unique homeomorphism $\Push'(\tilde{\gamma})=\Push(\gamma)$ where $\gamma$ is the projection of $\tilde{\gamma}$. We see that $\Push'$ is an isomorphism between $\RMCG(\OO,Q') \cup \ker(\Forget)$ and $\pi_1(\R^2 \sm \tilde{Q}, \tilde{q})$, where the latter is generated by simple loops around a single point in $\tilde{Q}$.

Applying the same approach as in the proof of Lemma~\ref{l:homotopy}, one proves that for every point $a$ in $\Pt$ there exists a simple loop $\tilde{\alpha}$ based at $\tilde{q}$, such that the bounded component of the complement of the loop contains $a$ and no other points from $\tilde{Q}$, which projects two-to-one to a simple loop $\alpha$ based at $q$ in $\OO$. Then $\alpha$ is inessential in $(\OO,P)$ and $$\Push'(\tilde{\alpha})=\Push(\alpha^2)=T_{\alpha}^2.$$ Similarly, for every point $b$ in $\tilde{Q} \sm \Pt$ there exists a simple loop $\tilde{\beta}$ based at $\tilde{q}$, such that the bounded component of the complement of the loop contains $b$ and no other points from $\tilde{Q}$, which projects one-to-one to a simple loop $\beta$ based at ${q}$ in $\OO$. Then $\beta$ is trivial in $(\OO,P)$ and $$\Push'(\tilde{\beta})=\Push(\beta)=T_{\beta}.$$ As $\pi_1(\R^2 \sm \tilde{Q}, \tilde{q})$ is generated by all possible curves $\alpha$ and $\beta$, the statement of the theorem follows.
\end{proof}

\begin{definition} Denote by $\Lift(\phi)$ the virtual endomorphism of $\PMCG(\OO,Q)$ that acts by lifting by $f$, i.e. we write $\Lift(\phi)=\psi$ whenever there exists $\psi \in \PMCG(\OO,Q)$ such that $\phi \circ f = f \circ \psi$.
\end{definition}

\begin{proposition} $\Lift(\phi)\colon \RMCG(\OO,Q) \to \RMCG(\OO,Q)$ is a well-defined endomorphism. If $f$ admits no Levy cycles, then for every $\phi \in \RMCG(\OO,Q)$, there exist an $n$ such that $\Lift^n(\phi)=\id$.
\label{p:nilpotent}
\end{proposition}

\begin{proof} It is enough to prove the statement for a generating set of $\RMCG(\OO,Q)$. By Theorem~\ref{t:rmcg} we only need to consider two cases.

\textbf{Case I.} Suppose $\phi = T_\alpha$ where $\alpha$ is a simple closed curve in $(\OO,Q)$, which is trivial in $(\OO,P)$. All connected components $\alpha_i$ of  $f^{-1}(\alpha)$ are pairwise disjoint simple closed curves that are trivial in $(\OO,P)$ and are mapped by $f$ to $\alpha$ with degree 1. It is straightforward to see that $$T_\alpha \circ f = f \circ \prod T_{\alpha_i}.$$ Thus $$\Lift(T_\alpha)=\prod T_{\alpha_i} \in \RMCG(\OO,Q)$$ is well-defined. Similarly, denote by $\alpha_i^n$ all connected components of $f^{-n}(\alpha)$; then $$T_\alpha \circ f^n = f^n \circ \prod T_{\alpha^n_i}\text{ and }\Lift^n(T_\alpha)=\prod T_{\alpha^n_i}.$$ By Corollary~\ref{c:FiniteDepth}, there exists an integer $n$ such that all $\alpha_i^n$ are inessential in $(\OO,Q)$, implying $$\Lift^n(T_\alpha)=\prod T_{\alpha^n_i}=\id.$$

\textbf{Case II.} Suppose $\phi=T_\beta^2$ where $\beta$ is a simple closed curve in $(\OO,Q)$, which is non-trivial and inessential in $(\OO,P)$. The interior of $\beta$ contains a unique critical value $p$ of $f$. All connected components $\beta_i$ of $f^{-1}(\beta)$ are pairwise disjoint simple closed curves that are inessential in $(\OO,P)$. Each $\Int(\beta_i)$ contains a unique $f$-preimage $p_i$ of $p$. If $p_i \in P$, then it is not a critical point of $f$ and $\beta_i$ is  mapped by $f$ to $\beta$ with degree 1. If $p_i \notin P$, then it is a critical point and $\beta_i$ is trivial in $(\OO,P)$ and is mapped by $f$ to $\beta$ with degree 2. As in the Case I, we see that  $$T_\beta^2 \circ f = f \circ (\prod_{p_i\in P} T_{\beta_i}^2 \circ \prod_{p_i \notin P} T_{\beta_i})$$ and $$\Lift(T_\beta^2)= \prod_{p_i\in P} T_{\beta_i}^2 \circ \prod_{p_i \notin P} T_{\beta_i} \in \RMCG(\OO,Q)$$ is well-defined. We also see that $$\Lift^n(T_\beta^2)= \prod_{p^n_i\in P} T_{\beta^n_i}^2 \circ \prod_{p^n_i \notin P} T_{\beta^n_i}$$ where $\beta^n_i$ are the connected components of $f^{-n}(\beta)$ and $p^n_i$ denote the corresponding $f^n$-preimages of $p$. By Corollary~\ref{c:FiniteDepth}, there exists an integer $n$ such that all $\beta_i^n$ are inessential in $(\OO,Q)$, implying $$\Lift^n(T_\beta^2)= \prod_{p^n_i\in P} T_{\beta^n_i}^2 \circ \prod_{p^n_i \notin P} T_{\beta^n_i}=\id.$$
\end{proof}

\begin{lemma} If $\psi =\Lift(\phi)$ for some $\phi \in \PMCG(\OO,Q)$, then  $f \circ \phi$ is Thurston equivalent to $f \circ \psi$.
\end{lemma}

\begin{proof}$f \circ \psi = \phi \circ f = \phi \circ (f \circ \phi) \circ \phi^{-1}.$
\end{proof}

We arrive at the following statement.

\begin{theorem}
If $f$ admits no Levy cycle and $\phi \in \RMCG(\OO,Q)$ then $f \circ \phi$ is Thurston equivalent to $f$.
\label{t:allEquivalent}
\end{theorem}

\begin{proof}
By proposition~\ref{p:nilpotent} and the previous lemma, there exists $n$ such that $f \circ \phi$ is equivalent to $f \circ \Lift^n(\phi) = f \circ \id = f$.
\end{proof}

We can now prove the first part of the statement of \thmref{th:degenerateLevy}.

\begin{theorem}  
\label{th:geometrization1}
 Let $(f,Q)$ be a Thurston $(2,2,2,2)$-map with postcritical set $P$ and marked set $Q\supset P$, such that $(f,P)$ is equivalent to a quotient $l$ of a real affine map $L(z)=Az+b$ by the orbifold group where both eigenvalues of $A$ are not equal to $\pm1$.  Then $(f,Q)$ is equivalent to a quotient of a real affine map by the action of the orbifold group if and only if $f$ admits no degenerate Levy cycle.
\label{t:hyperbolicCase}
\end{theorem}

\begin{proof}

\textbf{Necessity.} Suppose a quotient  $(l,Q)$ of a real affine map $L(z)=Az+b$ by the orbifold group $G$ admits a degenerate Levy cycle. By Proposition~\ref{p:NielsenLevy}, there exist two distinct points $q_1, q_2 \in Q$ in the same Nielsen class and Lemma~\ref{l:NielsenIndex} implies that there exist points $\tilde{q}_1,\tilde{q}_2$ in the fibers of $q_1,q_2$ respectively such that $$\text{ind}_{L,n}(\tilde{q}_1)=\text{ind}_{L,n}(\tilde{q}_2)=g\in G\text{, i.e. }L^n(\tilde{q}_i)=g(\tilde{q}_i)\text{ for }i=1,2.$$ Since $$L^n(z)=A^n z+b'\text{ and }g(z)=c \pm z$$ for some integer vectors $b'$ and $c$, the equation $$L^n(\tilde{q}_i)=g(\tilde{q}_i)$$ is equivalent to $$(A^n \pm I)z=c-b'$$ where $I$ denotes the identity matrix. By assumption, the eigenvalues of $A$ are not equal to $\pm 1$, hence the matrix $(A^n \pm I)$ is non-degenerate. This yields $\tilde{q}_1=\tilde{q}_2$, which is a contradiction.

\textbf{Sufficiency.} Suppose $f$ admits no Levy cycles and, hence, no two distinct points of $Q$ are in the same Nielsen class by Proposition~\ref{p:NielsenLevy}. Consider a universal cover $F$ of $f$; by Lemma~\ref{l:AffineLift} $$F(z)=L(z)=Az+b\text{ for all }z \in \Pt.$$  Pick a point $\tilde{q}$ in the fiber of a periodic point $q \in Q$ of period $n$. Let $s$ be a unique solution of the equation $$L^n(z)=\text{ind}_{F,n}(\tilde{q}) \cdot z.$$ We can push the point $\tilde{q}$ by  a $G$-equivariant homotopy $\Phi_t(z)\colon \R^2 \to \R^2$ 
along some path $\alpha$ in $\R^2 \sm \tilde{Q}$ that ends at $s$. Since $\Phi$ is $G$-equivariant, it pushes the point $$F^n(\tilde{q})=\text{ind}_{F,n}(\tilde{q}) \cdot \tilde{q}$$ along the path $\text{ind}_{F,n}(\tilde{q}) \cdot \alpha$ to the point $$\text{ind}_{F,n}(\tilde{q}) \cdot s = L^n(s).$$ Therefore,  for $F_1=\Phi_1 \circ F \circ \Phi_1^{-1}$, we have $F_1^n(s)=L^n(s)$. Let $s'=g \cdot s$, where $g\in G$ be any other point in the same fiber as $s$. Then $G$-equivariance of $\Phi$ implies
\begin{eqnarray*}
F_1^n(s')=\Phi_1 \circ F^n \circ \Phi_1^{-1} (g \cdot s)=\Phi_1 \circ F^n (g \cdot \Phi_1^{-1}(s))=\\ \Phi_1 (F_*^n g \cdot F^n \circ \Phi_1^{-1}(s))= 
F^n_*g \cdot \Phi_1 \circ F^n \circ \Phi_1^{-1}(s)=F^n_*g \cdot F^n_1(s) = F^n_*g \cdot L^n(s).
\end{eqnarray*}
Since $F=L$ on $\Pt$, their actions on the orbifold group  are the same: $F_*=L_*$. 
 Thus, 
$$ F_1^n(s')=F^n_*g \cdot L^n(s)=L^n_*g \cdot L^n(s) = L^n(g \cdot s)=L^n(s').
$$

We repeat this procedure for each periodic point in $Q$ to obtain a $G$-equivariant homotopy $\Psi_t(z) \colon \R^2 \to \R^2$ and set $F_2=\Psi_1 \circ F \circ \Psi_1^{-1}$, such that for any point $s=\Psi_1(\tilde{q})$, where $\tilde{q}$ is in the fiber of a periodic point of any period $n$ from $Q$, we have $F_2^n(s)=L^n(s)$. The only possible obstacle can occur when we need to push some point $\tilde{q}$ from the fiber of $q$ into the fiber of some other point $q'$. This would immediately imply that $q$ and $q'$ are in the same Nielsen class, which contradicts our assumptions.

 Note that our construction automatically implies $F_2(s)=L(s)$ for all $s=\Psi_1(\tilde{q})$, where $\tilde{q}$ is in the fiber of a periodic point $q$ of any period $n$. Indeed, if $F^n(z)=g \cdot z$, then $$F^n(F(z))=F(F^n(z))=F(g \cdot z)=F_*g \cdot F(z),$$ hence
$$ \text{ind}_{F,n}(F(z))= F_* \text{ind}_{F,n}(z)=L_* \text{ind}_{F,n}(z).$$ Therefore, if $s=\Psi_1(\tilde{q})$ is a unique solution of the equation $$L^n(z)=\text{ind}_{F,n}(\tilde{q}) \cdot z,$$ then $L(s)$ is a unique solution of the equation $$L^n(z)=\text{ind}_{F,n}(F(\tilde{q})) \cdot z=L_*\text{ind}_{F,n}(\tilde{q}) \cdot z,$$ because $$L^n(L(z))=L(L^n(z))=L(\text{ind}_{F,n}(\tilde{q}) \cdot z)=L_*\text{ind}_{F,n}(\tilde{q}) \cdot L(z).$$ This yields $\Psi_1(F(\tilde{q}))=L(s)$ and $$F_2(s)=F_2 \circ \Psi_1(\tilde{q})=\Psi_1 \circ F(\tilde{q})=L(s).$$

Now we perform an analogous  procedure on all strictly pre-periodic points. Let $q \in Q$ be a strictly pre-periodic point and $\tilde{q}$ be some point in its fiber. Denote by $n$ the pre-period of $q$, i.e. the smallest integer such that $f^n(q)$ is periodic. We find a $G$-equivariant homotopy that pushes $\tilde{q}$ to $L^{-n}(F_2^n(\tilde{q}))$ and leaves all fibers of other points of $Q$ in place. After repeating this process for all pre-periodic points of $f$, we construct a $G$-equivariant homotopy $\Xi_t(z)\colon \R^2 \to \R^2$ such that $F_3= \Xi_1 \circ F \circ \Xi_1^{-1}$ agrees with $L(z)$ on  $\Xi_1(\tilde{Q})$, in particular ${F_3}_*=L_*$. 

Denote by $f_3$ and $\xi$ the quotients of $F_3$ and $\Xi_1$ respectively by the action of $G$. Then $f_3= \xi_1 \circ f \circ \xi_1^{-1}$, and $(f,Q)$ is conjugate (and, hence, Thurston equivalent) to $(f_3, \xi(Q))$. Set $\Theta(z)=L^{-1} \circ F_3(z)$; we see that $$\Theta(g\cdot z)=L^{-1} \circ F_3(g \cdot z)=L^{-1} ({F_3}_* g \cdot F_3(z))= L^{-1} ({L}_* g \cdot F_3(z))=g \cdot L^{-1} \circ F_3(z),$$ i.e. $\Theta$ is $G$-equivariant. Therefore $f_3=l \circ\theta$ where $\theta$ is the quotient of $\Theta$ by the action of $G$. Since $F_3=L$ on $\tilde{Q}$, the universal cover $\Theta$ of $\theta$ is identical on $\tilde{Q}$ so $\theta \in \RMCG(\OO,Q)$. By Theorem~\ref{t:allEquivalent} $(f_3,\xi(Q))$ and $(l,\xi(Q))$ are Thurston equivalent, which concludes our proof.
\end{proof}

\subsection{Uniqueness}
We now prove the uniqueness part of the statement of \thmref{th:degenerateLevy}:
\begin{theorem}
\label{t:uniqueness}
Let $(l_i,Q_i)$ be a Thurston map that is a quotient of an affine map $L_i(z)=A_i z +b$ ($A_i\in \text{M}_2(\ZZ)$) by the action of an orbifold group $G$ for $i=1,2$. Suppose that eigenvalues of $A_i$ are not equal to $\pm1$ for $i=1,2$. If $(l_1,Q_1)$ and $(l_2,Q_2)$ are Thurston equivalent, then $(l_1,Q_1)$ and $(l_2,Q_2)$ are conjugate by a quotient of an affine map. In other words, there exist $g \in G$ and a real affine map $S$ 
with linear part in $\text{SL}_2(\ZZ)$ such that $L_2=g \cdot S \circ L_1\circ S^{-1}$ and $S$ sends $\tilde{Q}_1$ to $\tilde{Q}_2$. 
\end{theorem}

\begin{proof}

Let $\phi, \psi$ realize the Thurston equivalence of $(l_1,Q_1)$ and $(l_2,Q_2)$:

$$
\begin{diagram}
\node{(\OO,Q_1)} \arrow{e,t}{\psi} \arrow{s,l}{l_1} \node{(\OO,Q_2)}
\arrow{s,l}{l_2}
\\
\node{(\OO,Q_1)} \arrow{e,t}{\phi} \node{(\OO,Q_2)}
\end{diagram}
$$
where $\phi$ and $\psi$ are homotopic relative $Q_1$. Then there exist universal covers  $\tilde{\phi}$ and $\tilde{\psi}$ of $\phi$ and $\psi$ such that the following diagram commutes:

$$
\begin{diagram}
\node{(R^2,\tilde{Q}_1)} \arrow{e,t}{\tilde{\psi}} \arrow{s,l}{L_1} \node{(R^2,\tilde{Q}_2)}
\arrow{s,l}{L_2}
\\
\node{(R^2,\tilde{Q}_1)} \arrow{e,t}{\tilde{\phi}} \node{(R^2,\tilde{Q}_2)}
\end{diagram} 
$$

By Lemma~\ref{l:AffineLift}, both $\tilde{\phi}$ and $\tilde{\psi}$ are affine on $\Pt$. Since $\phi$ and $\psi$ are homotopic relative $P\subset Q_1$, there exists $g\in G$ such that $\tilde{\psi}=S$ and $\tilde{\phi}=g\cdot \tilde{\psi}=g \cdot S$ for all points in $\Pt$, where $S$ is a real affine map. Note that the linear part of $S$ has determinant $\pm1$ because $\phi$ and $\psi$ are one-to-one maps. Therefore $$L_2= \tilde{\phi} \circ L_1 \circ \tilde{\psi}^{-1}=g \cdot  \tilde{\psi} \circ L_1 \circ \tilde{\psi}^{-1}=g \cdot S \circ L_1\circ S^{-1}\text{ on }\Pt.$$ As both sides of the last equation are real affine, the equation holds for all points in $\R^2$. 

Replace $L_2$ by another universal cover $L_2'=g^{-1} \cdot L_2$ of $l_2$ so that $L'_2= S \circ L_1\circ S^{-1}$; set $\tilde{\phi}'=g^{-1} \cdot \tilde{\phi}$. Then both $\tilde{\phi}'$ and $\tilde{\psi}$ agree with $S$ on $\Pt$ and it follows that $\tilde{\phi}'$ and $\tilde{\psi}$ agree on $\tilde{Q}_1$ and  act by $\tilde{\phi}'_*=\tilde{\psi}_*=S_*$ on the first homology group of $\OO$. Consider a lift $\tilde{q}$ of a periodic point $q \in Q_1$ of some period $n$. Recall that $\tilde{q}$ is a unique solution of  $L_1^n(z)=g_1 \cdot z$, where $g_1=\text{ind}_{L_1,n}(\tilde{q})$. Then $$\tilde{\psi}(L_1^n(\tilde{q}))=\tilde{\psi}(g_1 \cdot \tilde{q})=S_*(g_1) \cdot \tilde{\psi}(\tilde{q}).$$ This yields $$\text{ind}_{L'_2,n}(\tilde{\psi}(\tilde{q}))=S_*(g_1)$$ and $\tilde{\psi}(\tilde{q})$ is a unique solution of $${L'_2}^n(z)=S_* (g_1)\cdot z,$$ which is equivalent to $$ S \circ L_1\circ S^{-1}(z)=S (g_1 \cdot S^{-1}(z))$$ or $$L_1 \circ S^{-1}(z)=g_1 \cdot S^{-1}(z).$$ We conclude that $\tilde{\psi}(\tilde{q})=S(\tilde{q})$ for all lifts of periodic points in $Q_1$. 

For a lift $\tilde{p}$ of a pre-periodic point $p \in Q_1$, consider some $k$ such that $L_1^k(\tilde{p})=\tilde{q}$, where $q=l_1^k(p)$ is periodic. Then 
$$ \tilde{\psi}(\tilde{p})={L'_2}^{-k}\circ \tilde{\psi} \circ L_1^k(\tilde{p})={L'_2}^{-k}\circ \tilde{\psi} (\tilde{q})={L'_2}^{-k} \circ S(\tilde{q})=S\circ L_1^k(\tilde{q})=S(\tilde{p}).
$$
We have shown that $S$ sends $\tilde{Q}_1$ to $\tilde{Q}_2$, therefore the quotient of $S$ to $\OO$ not only conjugates $l_1$ and $l_2$, but sends $Q_1$ to $Q_2$.
\end{proof}

\section{Constructive geometrization of Thurston maps with parabolic orbifolds}
\begin{theorem}
\label{thm:Parabolics}
  There exists an algorithm $\cA_9$ which for any marked Thurston map $f$ with a parabolic orbifold whose matrix does not have eigenvalues $\pm1$
finds either a degenerate Levy cycle or an equivalence to a quotient of an affine map with marked pre-periodic orbits.
\end{theorem}
\begin{proof}
The proof is completely analogous to the argument given in \cite{BBY}. We begin by identifying the orbifold group $G$ and finding
an affine map $L(x)=Ax+b$ such that $f$ without marked points is equivalent to the quotient $l$ of $L$ by $G$ (\thmref{t:4PointsCase}).

We now execute two sub-programs in parallel:
\\{\bf (I)} we use algorithm $\cA_8$ (\propref{enumerate3}) to enumerate all $f$-stable multicurves $\Gamma_n$. Using algorithm $\cA_2$
(\propref{homotopy check}) we check whether $\Gamma_n$ is a degenerate Levy cycle. If yes, we output {\bf degenerate Levy cycle found} and halt;
\\{\bf (II)} we identify all forward invariant sets $S_k$ of pre-periodic orbits of $l$ of the same cardinality as the set of marked
points of $f$. We use algorithm $\cA_6$ (\propref{enumerate1}) to enumerate the sequence $\psi_n$ of all elements of $\PMCG(S^2, Q)$.
For every $\psi_n$ and each of the finitely many sets $S_k$ we use algorithm $\cA_3$ (\propref{identify-isometry}) to 
check whether $h_k \circ \psi_n$ realizes Thurston equivalence between $f$ and $g$ with marked points $S_k$, where  $h_k \colon (S^2, Q) \to (S^2, S_k)$ is an arbitrary chosen homeomorphism. If yes, we output {\bf Thurston equivalence found}, list the maps $g$, $h_k \circ \psi_n$ and the set $S_k$ and halt.

By \thmref{th:degenerateLevy} either the first or the second sub-program, but not both, will halt and deliver the desired result. 
\end{proof}

\section{Constructive canonical geometrization of a Thurston map}
\begin{theorem}\label{th:canonicalgeometrization}
  There exists an algorithm which for any Thurston map $f$ finds its canonical obstruction $\Gamma_f$.

Furthermore, let $\cF$ denote the collection of the first return maps of the canonical decomposition
of $f$ along $\Gamma_f$. Then  the algorithm outputs the following information:
\begin{itemize}
\item for every first return map with a hyperbolic orbifold, the unique (up to M{\"o}bius conjugacy)
marked rational map equivalent to it;
\item for every first return map of type $(2,2,2,2)$ the unique (up to affine conjugacy) affine map of the form $z \mapsto Az+b$ where $A \in \text{SL}_2(\ZZ)$ and $b\in \frac{1}{2}\Z^2$ with marked points which is equivalent to $f$ after quotient by the orbifold group $G$;
\item for every first return map which has a parabolic orbifold not of type $(2,2,2,2)$ the unique (up to  M{\"o}bius conjugacy) 
marked rational map map equivalent to it, which is a quotient of a complex affine map by the orbifold group.
\end{itemize}
\end{theorem}

\begin{proof}
The result of \cite{BBY} together with \thmref{thm:Parabolics} implies the existence of the subprogram $\cP$ which given a marked Thurston map $f$
does the following:
\begin{enumerate}
\item if $f$ has a hyperbolic orbifold and is obstructed, it outputs 
a Thurston obstruction for $f$;
\item if $f$ has a parabolic orbifold not of type $(2,2,2,2)$ and a  degenerate Levy cycle it outputs such a Levy cycle;
\item if $f$ is a $(2,2,2,2)$ map such that the corresponding matrix has two distinct integer eigenvalues
 outputs a Thurston obstruction for $f$;
\item if $f$ is a $(2,2,2,2)$ map with a  degenerate Levy cycle outputs such a Levy cycle;
\item in the remaining cases output a geometrization of $f$ as described in the statement of the theorem.
\end{enumerate}
We apply the subprogram $\cP$ recursively to decompositions of $f$ along the found obstructions until 
no new obstructions are generated (this will eventually occur by Theorem~\ref{thm:CharacterizationCanonical} and \corref{c:generalobs}).

 Denote by $\Gamma$ the union of all obstructions thus generated. Use algorithm $\cA_2$ and sub-program $\cP$ to
find the set $S$ be a set of all subsets $\Gamma'\subset \Gamma$ such that:
\begin{itemize}
\item $\Gamma'$ is a Thurston obstruction for $f$;
\item denote $\cF'$ the union of first return maps obtained by decomposing along $\Gamma'$. Then no $h\in\cF'$ is a $(2,2,2,2)$ map 
whose matrix has distinct integer eigenvalues, and every $h\in\cF'$ which is not a homeomorphism is geometrizable. 
\end{itemize}
Set $$\Gamma_c\equiv \cap_{\Gamma'\in S}\Gamma'.$$
By Theorem~\ref{thm:CharacterizationCanonical} and \corref{c:generalobs}, $\Gamma_c$ is the canonical obstruction of $f$.
\end{proof}

\section{Partial resolution of the problem of decidability of Thurston equivalence}
Denote by $\cH$ the class of Thurston maps $f$ such that every first return map in the canonical decomposition of $f$ has  hyperbolic orbifold.
In this section we prove the following theorem:
\begin{thm}\label{th:deciding}
There exists an algorithm which given a PL Thurston map $f\in\cH$ and any PL Thurston map $g$ decides whether
$f$ and $g$ are equivalent or not.
\end{thm}
We will need several preliminary statements.

\begin{prop}
\label{p:uniq}
 If $(f, Q_f)$ and $(g,Q_g)$ are Thurston equivalent marked rational maps with hyperbolic orbifolds,
then the pair  $\phi, \psi$ realizing the equivalence
 $$ \phi \circ f = g \circ \psi $$ 
 is unique up to homotopy relative $Q_f$.
\end{prop}

\begin{proof}
  The statement is equivalent to saying that there are no non-trivial self equivalences of $f$. 
 If $ \phi \circ f = f \circ \psi $, where  $\phi$ and $\psi$ represent the same mapping class $h$, then $\sigma_f \circ h = h \circ \sigma_f$. If $\tau$ is the unique fixed point of $\sigma_f$, then $h(\tau)$ is also fixed, yielding a contradiction.

\end{proof}

For the following see \cite{Pil2}:

\begin{theorem}
Let $f$ and $g$ be two Thurston maps, and $\Gamma_f=\{\alpha_1,\ldots,\alpha_n\}$ and $\Gamma_g=\{\beta_1,\ldots,\beta_n\}$ be the corresponding canonical obstructions.
  Let $A_i, B_i$ be decomposition annuli homotopic to $\alpha_i,\beta_i$ respectively. If $f$ and $g$ are equivalent then there exists an equivalence pair $h_1,h_2$ such that $h_1(A_i)=B_i$ (up to a permutation of indexes) and $h_1$ on $\partial A_i$ is any given orientation-preserving
homeomorphism of the boundary curves.
\end{theorem}

\noindent
Recall that the components of the complement of all $A_i$ (resp $B_i$) are called  \emph{thick parts}.

\begin{corollary}
\label{c:thick}
 If $h_1$, $h_2$ are as above then each thick component $C$ is mapped by $h_1$ to a thick component $C'$. 
The components $C$ and $C'$ must have the same period and pre-period. 
When both are periodic, consider the patched components $\tl C$ and $\tl C'$  and consider the 
corresponding first-return maps  $\cF_{\tl C}$ and $\cF_{\tl C'}$. Then these maps are Thurston equivalent. 
\end{corollary}

\begin{proposition} If $f$ and $g$ are equivalent Thurston maps in standard form then there exists an equivalence pair $(h_1,h_2)$ such that $h_1$ and $h_2$ restrict to the identity map on all $\partial A_i$ and $h_1$ is homotopic to $h_2$ on each thick component relative $\partial A_i$ and $Q_f$. 
\end{proposition}
\begin{proof}
As was shown above, there exists an equivalence pair $(h_1,h_2)$ of  $f$ and $g$ that descends to 
an equivalence of respective canonical decompositions.  More precisely, there is a correspondence between thick 
components of $f$ and thick components of $g$ which conjugates the component-wise action of $f$ to the action of $g$ such 
that the first return maps of  corresponding periodic components are Thurston equivalent.    
We fix coordinates on $A_i$ and $B_i$ and can chose $h_1$ and $h_2$ such that both restrict to
 the identity map on all $\partial A_i$. Since  $h_1$ and $h_2$ are homotopic relative $Q_f$ and coincide on  
$\partial A_i$, restricted to each thick component of $f$ the two homeomorphisms can differ 
(up to homotopy relative the boundary of the component and the marked set) 
only by a composition of some powers of Dehn twists around the boundary components. 
Pushing this Dehn twists inside the annuli $A_i$, we can further normalize the pair $(h_1,h_2)$ so 
that $h_1$ is homotopic to $h_2$ on each thick component. 
\end{proof}


The following is standard (see e.g. \cite{primer}):
\begin{prop}
\label{prop:free abelian}
For every Thurston obstruction $\Gamma=\{\alpha_1,\ldots,\alpha_n\}$, the Dehn twists
$T_{\alpha_j},\; j=1\ldots n$ generate a free Abelian subgroup of $\PMCG(S\setminus Q_f)$.
\end{prop}
We write $\ZZ^\Gamma\simeq \ZZ^n$ to denote the subgroup generated by $T_{\alpha_j}$.
 
\begin{proposition} If $f$ and $g$ are equivalent Thurston maps, then there exists an equivalence pair $(h_1,h_2)$ such that $h_1=h_0 \circ m$ where $h_0$ is constructed as above and $m \in \ZZ^\Gamma$.
\label{p:NormalEquivalence}
\end{proposition}
\begin{proof}
Consider an equivalence normalized as in the previous proposition. For every periodic thick component of $f$ with first return map that has hyperbolic orbifold, the restriction of $h_1$  to that component (after patching) will represent the unique (by Proposition~\ref{p:uniq}) mapping class that realizes Thurston equivalence to the corresponding periodic thick component of $g$. Under the assumptions of Theorem~\ref{th:deciding}, the normalized equivalence can be defined in this manner on all periodic thick components. 
This in turn defines $h_2$, and thus $h_1$, uniquely up to homotopy  relative $\partial A_i$ and $Q_f$ on every thick component that is  a preimage of a periodic thick component by pulling back $h_1$ by $f$. Repeating the pullback procedure we can recover $h_1$ on all thick components in the decomposition of $f$. Therefore using the decomposition data we can construct a mapping class $h_0$ which is homotopic to $h_1$ on all thick components and defined arbitrarily on $A_i$.  The restriction of $m=h_0^{-1}\circ h_1$ to every thick component is homotopic to the identity and the restriction of $m$  to every annulus $A_i$ is some power of the corresponding Dehn twist $T_{\alpha_i}$, i.e. $m \in \ZZ^\Gamma$. 

\end{proof}

Notice that by construction if $h_1 \circ f = g\circ h_2$ where $h_1= h_0 \circ m$ for some $m_1 \in \ZZ^\Gamma$, then $h_2$ is homotopic to $h_0 \circ m_2$ for some other $m_2 \in \ZZ^\Gamma$. Since we cannot check all elements of $\ZZ^\Gamma$ we will require the following proposition.

\begin{proposition}
\label{p:finiteN} There exists explicitly computable $N \in \N$ such that if $n\in \ZZ^\Gamma$ where all coordinates of $n$ are divisible by $N$, then $(h_0\circ (m_1+n)) \circ f = g\circ h_1$, with $h_1$ homotopic to $h_0\circ( m_2+M_\Gamma n)
$ rel $Q_f$, whenever $(h_0\circ m_1) \circ f = g\circ h_2$, with $h_2$ homotopic to $h_0\circ m_2$ rel $Q_f$.
\end{proposition}

\begin{proof}
We can take $N$ to be the least common multiple of all degrees of $f$ restricted to preimages of the annuli $A_i$. Then $n$ lifts through $f$ to $M_\Gamma n$ and we have the following commutative diagram:

\[
\begin{CD}
S^{2} @>M_\Gamma n>> S^{2} @>h_1>> S^{2}\\
@VVfV @VVfV @VVgV\\
S^{2} @>n>> S^{2} @>h_0\circ m_1>> S^{2}
\end{CD}
\]
\end{proof}

We can now present the proof of \thmref{th:deciding}:
\begin{proof}
The following  algorithm solves the problem. 

\begin{enumerate}
\item Find the canonical obstructions $\Gamma_f=\{\alpha_1,\ldots,\alpha_n\}$ and $\Gamma_g=\{\beta_1,\ldots,\beta_n\}$ 
 (\thmref{th:canonicalgeometrization}). 
\item Check whether the cardinality of the canonical obstructions $\Gamma_f=\{\alpha_1,\ldots,\alpha_n\}$ and $\Gamma_g=\{\beta_1,\ldots,\beta_n\}$  is the same, and whether Thurston matrices coincide. If not,
output {\bf maps are not equivalent} and halt. 
\item Construct decomposition annuli $A_i$ and $B_i$ as above. Geometrize the first return maps of patched thick parts 
(\thmref{th:canonicalgeometrization}). 
\item {\bf for all }$\sigma\in S_n$ {\bf do}
\item Is there a homeomorphism $h_\sigma$ of $S^2$ sending $A_i\to B_{\sigma(i)}$? If not, {\bf continue.}  Check that the component-wise dynamics of $f$ and $g$ are conjugated by $h_\sigma$. If not, {\bf continue.}

\item Construct equivalences between first return maps $\cF_i$ and $\cG_i$ of periodic thick components corresponding by $h_\sigma$. If the maps of some pair are not equivalent, {\bf continue.}

\item For all thick components $C_j^f$ check whether the Hurwitz classes of the patched coverings $$\tl f:\wtl{C_j^f}\to \wtl{f(C_j^f)}\text{ and }\tl g:\wtl{h_\sigma(C_j^f)}\to \wtl{g(h_\sigma(C_j^f))}$$ are the same (\thmref{th:hurwitz}).
 If not, {\bf continue.}

\item Can the equivalences between first return maps $\cF_i$ and $\cG_i$ constructed at step (6) be lifted 
via branched covers $\tl f$ and $\tl g$ to every sphere in the cycle (\thmref{th:hurwitz})? If not, {\bf continue.}

\item Check if the lifted equivalences preserve the set of marked points. If not, {\bf continue.}

\item Lift the equivalences, to obtain 
a homeomorphism  $h_0$ defined on all thick parts. 

\item Pick some initial homemorphisms $a_i:A_i \to B_i$ so that the boundary values agree with already defined boundary values of $h_0$. This defines $h_0$ on the whole sphere. 

\item {\bf for all } $n \in \ZZ^\Gamma$ with coordinates between 0 and $N-1$, where $N$ is as in Proposition~\ref{p:finiteN} {\bf do}

\begin{enumerate}
\item Try to lift $h_0 \circ n$ through $f$ and $g$ so that $(h_0 \circ n) \circ f =g \circ h_2$.  If this does not work, {\bf continue.}
\item By the discussion above $h_2=h_0 \circ m$ with $m \in \ZZ^\Gamma$. Compute $m$.

\item Find a solution in $\ZZ^\Gamma$ of the equation $Nx+n=M_\Gamma Nx +m$. If there is no integer solution,  {\bf continue.}
\item Output {\bf maps are equivalent} and $h_0 \circ (Nx+n)$ halt.
\end{enumerate}
\item {\bf end do}
\item {\bf end do}
\item output {\bf maps are not equivalent} and halt.
\end{enumerate}

If the algorithm outputs $h_0 \circ (Nx+n)$ at step 9(c), then by Proposition~\ref{p:finiteN} $h_0 \circ (Nx+n)$ lifts through $f$ and $g$ to a map which is homotopic to  $h_0 \circ (M_\Gamma Nx +m)= h_0 \circ (Nx+n)$ producing an equivalence between $f$ and $g$. If the algorithm fails to find an equivalence pair in this way, then Proposition~\ref{p:NormalEquivalence} implies that $f$ and $g$ are not equivalent.

\end{proof}

\section{Concluding remarks}
In this paper the problem of algorithmic decidability of Thurston equivalence of two Thurston maps $f$ and $g$
is resolved partially, when the decomposition of $f$ (or $g$) does not contain any parabolic elements or homeomorphisms.
Note that if the first return map $\cF$ of a periodic component $\tl S$ of the canonical decomposition of $f$ is a homeomorphism,
then the problem of equivalence restricted to $\tl S$ is the congugacy problem in $\MCG(\tl S)$. By \thmref{solvable-conjugacy-problem}, it can be resolved algorithmically. 

By \thmref{t:uniqueness}, in the case when $\cF$ is parabolic, Thurston equivalence
problem restricted to $\tl S$ reduces to a classical conjugacy problem of integer matrices:

\medskip
\noindent
{\sl Are two matrices in $\text{M}_2(\ZZ)$ conjugate by an element of $\text{SL}_2(\ZZ)$?}

\medskip
\noindent
This problem is solvable algorithmically as well (see e.g. \cite{Grun}).

Thus in both exceptional cases, we can constructively determine whether the first return maps of the thick parts in the
decompositions of $f$ and $g$ are Thurston equivalent or not. 
However, in contrast with \propref{p:uniq}, in this  case the homeomorphism realizing  equivalence  is not unique. 
This poses an obvious difficulty with 
checking whether $f$ is equivalent to $g$, as we have to check not one, but all possible
equivalences of parabolic and homeomorphic components of the decomposition. In other words, the homeomorphism $h_0$ constructed in the proof of Theorem~\thmref{th:deciding} is no longer unique; instead we get a certain subgroup of the Mapping Class Group of possible candidates. Extending our proof of decidability of Thurston equivalence to this case is an interesting problem, 
which may require, in particular, an algorithm for computing this subgroup.

\bibliographystyle{alpha}
\bibliography{biblio}

\end{document}